\documentclass[a4,12pt]{article}

\usepackage[all]{xy}
\usepackage{algorithmic}
\usepackage{amsthm}
\usepackage{amsmath}
\usepackage{amssymb}
\usepackage{animate}
\usepackage{color}
\usepackage{calligra,mathrsfs}
\usepackage{fancyhdr}
\usepackage{footmisc}
\usepackage{fancyvrb}
\usepackage{float}
\usepackage{graphicx}
\usepackage[english]{babel}
\usepackage{listings}
\usepackage{multicol}
\usepackage{mathrsfs}
\usepackage{mathtools}
\usepackage{pifont}
\usepackage{subfig}
\usepackage[subfigure]{tocloft}
\usepackage{tkz-graph}
\usepackage{tikz-cd}
\usepackage{tikz}
\usepackage[utf8]{inputenc}
\usepackage{verbatim}

\usetikzlibrary{arrows}
\usetikzlibrary{calc}
\usetikzlibrary{matrix}
\usetikzlibrary{trees}
\theoremstyle{plain}

\def\Pb{{\mathbb P}}
\def\Fb{{\mathbb F}}
\def\Zb{{\mathbb Z}}

\def\Wc{{\mathcal W}}
\def\Xc{{\mathcal X}}
\def\Yc{{\mathcal Y}}

\def\C~{\widetilde{C}}
\newtheorem{theorem}{Theorem}[section]
\newtheorem{definition}[theorem]{Definition}
\newtheorem{lemma}[theorem]{Lemma}

\newtheorem{corollary}[theorem]{Corollary}
\newtheorem{solution*}{Solution}
\newtheorem{proposition}[theorem]{Proposition}
\newtheorem{remark}[theorem]{Remark}

\numberwithin{equation}{section}

\newcommand{\PP}{\Pb^1\times\Pb^1}
\newcommand{\bmu}{\scriptsize $\bar{\mu}$}

\newcommand{\s}[1]{\scriptsize$#1$} 
\newcommand{\ti}[1]{\tiny$\bar{#1}$}

\DeclareMathOperator{\Hom}{\mathscr{H}\text{\kern -3pt {\calligra\large om}}\,}

\DeclarePairedDelimiter\ceil{\lceil}{\rceil}

\title{\bf Classification of Del Pezzo Orders with Canonical Singularities}
\author{ A. Nasr\thanks{Carleton University, School of Mathematics and Statistics,
      Ottawa, ON, Canada,\, Email: amir.nasr@carleton.ca}  \qquad
    }

\date{\today}

\begin{document}
\maketitle

\abstract{
We classify del Pezzo non-commutative surfaces that are finite over their centres and have no worse than canonical singularities. Using the minimal model program, we introduce the minimal model of such surfaces. We first classify the minimal models and then give the classification of these surfaces in general.  This presents a complementary result and method to the classification of del Pezzo orders over projective surfaces given by Chan and Kulkarni in 2003.}

\section{Introduction}
In modern algebraic geometry, non-commutative algebraic varieties are considered.  One open problem concerning non-commutative varieties is classification of non-commutative surfaces.  A nice subclass of non-commutative surfaces is the ones which are finite over their centres. Such a non-commutative surface is called an order or more precisely an order over its centre.

In the present work, we are interested in classifying  del Pezzo orders with canonical singularities. A del Pezzo order is a generalization of the notion being del Pezzo for commutative surfaces  that was introduced by  Pasquale del Pezzo in 1887. These are the surfaces with ample anti-canonical bundles.  In 2003 Chan and Kulkarni showed that if an order over a nice enough centre is del Pezzo, then the centre is a del Pezzo surface, \cite{chan2003pezzo}. Using this result, they classified del Pezzo orders over normal Gorenstein projective surfaces. After that, Chan and Ingalls generalized the terminology of minimal model program to orders  and classified minimal terminal orders, \cite{CI2}.  Using the fact that blowups of orders with terminal singularities will only have terminal singularities, we find all blowups of the minimal models which preserve the canonical bundle being ample. From now on, we call orders with terminal (canonical) singularities terminal (canonical) orders. Further, we write a terminal del Pezzo order as TdPO and a canonical del Pezzo order as CdPO.

\begin{theorem}
Let $\Xc$ be a TdPO over a smooth surface $S$ with ramification divisors $D=\cup D_i$ and  ramification degrees $e_i$. Then one of the following occurs
\begin{enumerate}
\item
$S=\Pb^2$, $3\leq\deg D\leq 5$. Ramification degrees $e_i$ are all equal, say $e$. Further, $e=2$ when $\deg D=5$ and $e=2$ or $e=3$ when $\deg D=4$
\item 
There is a sequence of blowups $f:S\rightarrow \Pb^2$ at a set of points $\Sigma=\{p_1, \cdots, p_n\}\subset \Pb^2$, where
\begin{itemize}
 \item 
 $\deg f_*D=3$,  $n<9$, and $\Sigma\subset f_*D$ is in  general position; or
 \item
 $\deg f_*D=3$, $n=1$, $p_1\notin f_*D$, and $e_i=2$; or
 \item
 $\deg f_*D=4$, $n=1$, $p_1\in f_*D$, and $e_i=2$.
 \end{itemize}
 \item 
 $S=\PP$ and if $D=aC_0+bF$ for crossing fibres $C_0$ and $F$, then $2\leq a,b\leq3$ and ramification degrees are all equal and, unless $D\sim 2C_0+2F$, they are all $2$.
 \item
 There is a sequence of blowups $f:S\rightarrow \PP$ at a set of points $\Sigma=\{p_1, \cdots, p_n\}\subset \PP$, where $f_*D$ has bi-degree $(2,2)$,  $n<8$, and $\Sigma\subset f_*D$ is in  general position.
\end{enumerate}
\end{theorem}

The  work aims to classify CdPOs. To do so, we first need to find their resolutions to terminal orders, we will see that the resolution of  a canonical del Pezzo order (which is a terminal order) is not necessarily del Pezzo, and further, their minimal models are more various.

\begin{theorem}
Let  $\Xc$ be a  CdPO over $Z$. Also, let $f:\Yc\rightarrow\Xc$  be a minimal resolution of $\Xc$ to a terminal almost del Pezzo order $\Yc$ and let $g:\Yc\rightarrow\Wc$ be a contraction of $\Yc$ to a minimal terminal almost del Pezzo order $\Wc$. Then  $\Wc$ has centre  $Z=\Pb^2$ or $Z=\Fb_n$ for $n=0,1,$ or $2$.
\end{theorem}
With the same procedure as TdPOs, we classify minimal  terminal almost del Pezzo orders (TAdPOs), and then  we find all the blowups which preserve being del Pezzo. 

\begin{theorem}
 Let $\Wc$ be a minimal TAdPO over $Z$ with ramification divisor $D=\cup D_i$ and ramification degrees $e_i$. Also let $g:\Yc\rightarrow \Wc$ represent $n$ iterated blowups of $\Wc$ at the points $\Sigma\subset Z$. Then the followings give a complete list of  $K_{\Yc}$-zero curve $E$'s such that if $f:\Yc\rightarrow\Xc$ contracts $E$, then $\Xc$ is a CdPO.
 \begin{enumerate}
 \item
 $Z=\Pb^2$, $\deg D=3$, $\Sigma$ is in almost general position (see Definition \ref{almost general position P2}) and we have one of the followings
 \begin{itemize}
  \item 
  $\Sigma$ is the set of a double infinitely near point with the exceptional curves $E_1$ and $E_2$, where $E_1^2=-2$ and $E_2^2=-1$; $E:=E_1$.
  \item
  $\Sigma=\{p, q\}$ where $e=2$, $p\in D$ and $q\notin D$; $E$ is the strict transform of the  line going through $p$ and $q$.
 \item
  $\Sigma$ contains  $3$ points in $D$ (infinitely near points are allowed) where there is a line  $\ell$ with multiplicity $3$ at $\Sigma$; $E:=\ell$.
 \item
  $\Sigma$ contains  $6$ points in $D$ (infinitely near points are allowed) where there is a conic  $C$ with multiplicity $6$ at $\Sigma$; $E:=C$.
  \item
 $\Sigma$ contains  $8$ points in $D$ (infinitely near points are allowed) where there is a nodal cubic  $C'$ with multiplicity $9$ at $\Sigma$; $E:=C'$.
  \end{itemize}
 \item
 $Z=\Pb^2$, $\deg D=4$,  $\Sigma$ is in almost general position and contains  $2$ or $3$ points in $D$ (infinitely near points are allowed) where there is a line  $\ell$ with multiplicity $2$ at $\Sigma$; $E:=\ell$. If $\Sigma$ has $3$ points, they are not collinear.
 
  \item
 $Z=\PP$,  and we have one of the followings
 \begin{itemize}
 \item
$D\equiv2C_0+2F$, $e=2$, and $\Sigma=\{p\}$ is a single point, where $p\notin D$. Then $E$ is the proper transform of any fibre (in any direction) passing through $p$.
\item
$D\equiv3C_0+2F$, $e=2$, and $\Sigma=\{p\}$ is a single point, where $p\in D$. Then $E$ is the proper transform of  any fibre in $[F]$ passing through  $p$.
\item
$D\equiv3C_0+3F$, $e=2$, and $\Sigma=\{p\}$ is a single point, where $p\in D$. Then $E$ is the proper transform  of any fibre (in any direction) passing through  $p$.
\item
$D\equiv2C_0+2F$, $e>1$, and $\Sigma\subset D$ is a set of points in almost general position (see Definition \ref{almost general position P1}) . Then $E$ is the blowup of  any curve in  $\Sigma$-almost general position.
  \end{itemize}
  \item
 $Z=\Fb_1$,  $\Sigma$ is in almost general position (see Definition \ref{almost general position F1}), $D\equiv 2C_0+4F$ where $C_0$ is the unique section with $C_0^2=-1$, $e=2$, and $E$ is one of the followings
 \begin{itemize}
 \item
$C_0$.
\item
$\tilde{F}$ , where $F$ is a fibre and  multiplicity of $F$ at $\Sigma$ is $2$.
\item
An exceptional curve with self-intersection equals $-2$.
  \end{itemize}
  \item
 $Z=\Fb_2$,  $\Sigma$ is in almost general position (see Definition \ref{Z2 Y adp}), $D\equiv 2C_0+4F$ where $C_0$ is the unique section with $C_0^2=-2$, $e$ is free and we have one of the followings
 \begin{itemize}
 \item 
  The section $C_0$
 \item
  Blowing up  points $p\notin D$; $E:=F$ where $F$ is the fibre passing $p$.
  \item
  Blowing up  a set of points $\Sigma=\{p_1,\cdots, p_n\}\subset D$ in almost general position where $n\leq7$;  $E:=F$ is any fibre with multiplicity $2$ at $\Sigma$.
 \item
 $E:=E_j$ where $E_j$ is $(-2)$-exceptional curve.
  \end{itemize}
\end{enumerate}
\end{theorem}

We always work with an algebraically closed field. $\Xc$, $\Yc$, and $\Wc$ denote maximal orders over the centres $Z_{\Xc}$, $Z_{\Yc}$ and $Z_{\Wc}$, respectively. A maximal order $\Xc$ over $Z_{\Xc}$ correspond to a finite number of ramified divisors $D_{i, \Xc}\subset Z_{\Xc}$ their  sum $D_{\Xc}=\sum_iD_{i, \Xc}$ is called ramification divisors. When it is clear we show the centre of $\Xc$ by $Z$, a ramified divisor by $D_i$ and its ramification divisors by $D$. Also for each ramified divisor $D_i$, there is a ramification degree $e_i$. These give us the divisor
\[
\Delta_{\Xc}:=\sum_i\left(1-\dfrac{1}{e_i}\right)D_{i, \Xc},
\]
called the discriminant of $\Xc$ and again when it is clear we show it by $\Delta$.

 This work is based on my PhD thesis under the supervision of Dr. Colin Ingalls.

\section{Classification of TdPOs}\label{section2}
In this section we will classify TdPOs.  Chan and Ingalls in \cite{CI2}, generalized the notion of minimal model program (MMP) to terminal orders over surfaces. Running MMP on a TdPO, we get a minimal TdPO. We firstly classify minimal TdPOs and then using this classification, we can classify all TdPOs.

\subsection{Minimal TdPOs}
\begin{definition}\label{minimal terminal order}
Let $\Xc$ be a terminal order over $Z$ and let $K_{\Xc}=K_Z+\Delta$ be its canonical divisor. Then $\Xc$ is a minimal terminal order if for every irreducible curve $C\in Z$, either $K_{\Xc}C\geq 0$ or $C^2\geq 0$.
\end{definition}

\begin{proposition}\cite[Theorem 3.10]{CI2}\label{contracting terminal orders is terminal}
 Let $\Xc$ be a terminal order over $Z$. Suppose there is an irreducible curve $E\in Z$ such that $E^2<0$ and $K_{\Xc}.E<0$. Then there exists a map $\pi:Z\rightarrow Z'$ that contracts exactly $E$ and the order $\Xc'$ over $Z'$ is terminal.
\end{proposition}

Let  $K_{\Xc}$ be not nef. If for an irreducible curve $E$ the self intersection is negative, then $\pi:Z\rightarrow Z'$ contracts $E$. Proposition \ref{contracting terminal orders is terminal} proves that the order $\Xc'$ over $Z'$ is a terminal order. Then $\Xc$ can be replaced by $\Xc'$ and we can repeat the proposition for $\Xc'$. This ensures that we end with a minimal terminal order.

\begin{corollary}\label{centre of minimal terminal order}
Let $\Xc$ be a terminal order over $Z$. There exists a sequence of blowdowns of $(-1)$-curves 
\[
f:Z\rightarrow Z_1\rightarrow\cdots\rightarrow Z_n=Z'
\]
and a maximal order $\Xc'$ over $Z'$ where $K_{\Xc'}=K_{Z'}+\Delta'$. Then one of the followings holds,
\begin{itemize}
\item $K_{\Xc'}$ is nef.
\item $\pi:Z'\rightarrow C$ is a ruled surface and $-K_{\Xc'}.F>0$ for a fibre $F$. Further, $Z'$ contains no irreducible  curve $C$ such that $C^2<0$ and $K_{\Xc'}.C<0$.
\item $Z'\simeq \Pb^2$ and $-K_{\Xc'}$ is ample.
\end{itemize} 
\end{corollary}

\begin{lemma}\label{contraction of del pezzo orders}
Let $\Xc$ be a del Pezzo order over $Z$ and let $f:Z\rightarrow Z'$ be a birational morphism  which contracts exactly an irreducible curve $E$ such that $E^2=-1$. Then the order $\Xc'$ over $Z'$ is del Pezzo.
\end{lemma}

\begin{proof}
Consider the equations
\begin{align*}
K_{\Xc}&=K_Z+\Delta\\
K_{\Xc'}&=K_{Z'}+\Delta'\\
K_{Z}+\Delta&\equiv f^*(K_{Z'}+\Delta')+aE.
\end{align*}
As $\Xc$ is a del Pezzo order, $(K_{Z}+\Delta)^2>0$ and $(K_{Z}+\Delta).C<0$ for any effective curve $C\in Z$. Thus
\begin{align*}
(K_{Z}+\Delta)^2&=(f^*(K_{Z'}+\Delta')+aE)^2\\
&=(K_{Z'}+\Delta')^2-(a)^2>0,
\end{align*}
so  $(K_{Z'}+\Delta')^2>0$.

If $C'$ is an effective curve in $Z'$, then $C=f^*C'$ is an effective curve in $Z$.
\begin{align*}
0>(K_{Z}+\Delta)C&=(f^*(K_{Z'}+\Delta')+aE)C\\
&=f^*(K_{Z'}+\Delta')C+aE.C\\
&=(K_{Z'}+\Delta')f_*C+0\\
&=(K_{Z'}+\Delta')C'.
\end{align*}
\end{proof}

Now if we add the assumption of ampleness of the anti-canonical bundle $-K_{\Xc}$ to Corollary  \ref{centre of minimal terminal order} we get the following result.

\begin{corollary}\label{centre of minimal terminal del pezzo orders}
Let $\Xc$ be a minimal TdPO over $Z$. Then $Z=\PP$ or $Z=\Pb^2$.
\end{corollary}

So all TdPOs are blowups of  the minimal TdPOs over $\PP$ or $\Pb^2$. Therefore, for classification we first classify the ones over $\Pb^2$ and $\PP$ and then we blow them up. We classify such blowups that keep the order del Pezzo.

 When we blo wup an order at a point $p$ the canonical bundle of the new order depends on $p$ and the ramification divisors.
 
\begin{lemma}\cite[p:21]{CI2}\label{coefficient of E}
Let $\Xc$ be a terminal order over $Z$ and let $f:Z'\rightarrow Z$ be a blowup at a point $p$. Then we have the equation
\[
K_{Z'}+\Delta'\equiv f^*(K_{Z}+\Delta)+aE,
\]
where
\begin{enumerate}
\item
$a=1$, if $p$ is not in $D=\ceil{\Delta}$.
\item
$a=\dfrac{1}{e}$, if $p$ is a smooth point of $D$, where $e$ is the ramification degree of $D$.
\item
$a=\dfrac{1}{ne}$, if $p$ is a singular point of $D$, where the ramification degrees of the ramification curves crossing at $p$ are $e$ and $ne$.
\end{enumerate}
\end{lemma}

\begin{proposition}\label{decrease of K_X^2}
Let $\Xc$ be a terminal order over $Z$ and let $\Yc$ be the terminal order obtained by the blowup $f:Z'\rightarrow Z$ at a point $p$. Then
\[
K_{\Yc}^2= K_{\Xc}^2-a^2,
\]
where $a$ is as Lemma \ref{coefficient of E}.
\end{proposition}

\begin{proof}
We have $K_{\Yc}\equiv f^*(K_{\Xc})+aE$, therefore $$K_{\Yc}^2= (f^*(K_{\Xc})+aE)^2=f^*(K_{\Xc})^2+2af^*(K_{\Xc}).E-a^2=K_{\Xc}^2-a^2$$
\end{proof}

Before classifying TdPOs, we state some useful results which make our calculations easier. 
In this section and Section \ref{Minimal Models of CdPOs}, we will see that for any order $\Xc$ that we work with all the degrees are equal, say to $e$, and there is an effective divisor $M$ such that the ramification divisor $D\sim -K_Z+M$. In particular $K_{\Xc}=\frac{1}{e}(K_Z+(e-1)M)$. Then by Lemma \ref{coefficient of E} we get the following proposition.
\begin{proposition}\label{D=Z+M}
Let $\Xc$ be a terminal order over $Z$ with ramification divisor $D$ and assume there exists an effective divisor $M$ such that $D\sim -K_Z+M$. If $\Yc$ is the terminal order obtained from $\Xc$ by a blowup $f:Z'\rightarrow Z$ at points in the ramification locus. Then $K_{\Yc}=\frac{1}{e}\left(K_{Z'}+(e-1)f^*M\right)$
\end{proposition}

In Proposition \ref{D=Z+M}, if $M=0$ for some terminal order $\Xc$, i.e. the ramification divisor of $\Xc$ is anti-canonical. Then $K_{\Xc}=\dfrac{1}{e}K_Z$ and we get the following result.

\begin{proposition}\label{Z iff X}
Let $\Xc$ be a terminal order over $Z$ with anti-canonical ramification divisors, meaning that $D\sim -K_Z$, where all the ramification degrees are equal, say to $e$. Then $K_{\Xc}=\dfrac{1}{e}K_Z$ and $\Xc$ is del Pezzo if and only if $Z$ is del Pezzo. Furthermore, blowing up points of the ramification locus leaves the ramification anti-canonical.
\end{proposition}
\begin{proof}
The first part of the proposition is straightforward and left to the reader. The second part is a specific case of Proposition \ref{D=Z+M} where $M=0$.

\end{proof}


\begin{proposition}\label{Y dP M}
Let $D\subset Z$ be ramification divisors of an order and let $f:Z'\rightarrow Z$ be blowups at  $\Sigma\subset D$. If $\Yc$ is del Pezzo over $Z'$ with ramification divisors $D'\sim -K_{Z'}+M$ for some effective divisor $M$, then the following condition holds. Let $C\subset Z'$ be an effective curve. Then the multiplicity $m=mult_{\Sigma}f_*C<2-(e-1)M.f_*C+f_*C^2$.
\end{proposition}
\begin{proof}
The genus formula and Proposition \ref{D=Z+M} give
\begin{align*}
eK_{\Yc}.C=K_{Z'}.C+(e-1)f^*M.C&=-2-C^2+(e-1)M.f_*C\\&=-2-f_*C^2+m+(e-1)M.f_*C.
\end{align*}
\end{proof}

\begin{remark}\label{effective cone}
The necessary criteria in Proposition \ref{Y dP M} is exactly sufficient and need only be checked for generators $C$ of effective cone of $Z'$. Considering classification of del Pezzo surfaces in \cite[Theorem 3.4]{hidaka1981normal}, for any  del Pezzo surface $Z'$, the degree $\deg Z'\leq 9$. If $\deg Z'=9$, then $Z'=\Pb^2$, where the effective cone is generated by $H$, for a line $H$. If $\deg Z'=8$, then $Z'=\PP, \Fb_1$, or $\Fb_2$ and in these cases the effective cone is generated by $0$, $(-1)$, and $(-2)$-curves.  For del Pezzo surfaces with degree $\leq7$  the effective cone  is generated by $(-1)$ and $(-2)$-curves, see \cite[Theorem 3.10]{derenthal2008nef}.
\end{remark}

\subsection{TdPOs Over $\Pb^2$ and Its Blowups }
Let $\Xc$ denote a del Pezzo order over $\Pb^2$ and $D=\cup D_i$ denotes the ramification divisors of $\Xc$ with ramification degrees $\{e_i\}$. Using the results in \cite{CI2} and \cite{chan2003pezzo}, we have $3\leq \deg(D)\leq 5$. 
Additionally for ramification divisors $D=\cup D_i$,  all the ramification degrees $e_i$ are equal.
 
\begin{table}[H]
    \centering
\begin{tabular}{c|c}
$deg(D)$ & $e$ \\
\hline
$3$ & $\geq 2$\\
$4$ & $2, 3$\\
$5$ & $2$
\end{tabular}
\caption{Ramification divisors and degrees of terminal orders over $\Pb^2$}
\label{table}
\end{table}

Let $\Xc$ be a terminal order over $\Pb^2$ and let $D$ be ramification divisors  with the ramification degree $e$. Then,
\[
(K_{\Pb^2}+\Delta_{\Pb^2})^2=\left(-3H+\left(1-\frac{1}{e}\right)dH\right)^2.
\]
Thus
\begin{equation}\label{self intersect P^2}
(K_{\Pb^2}+\Delta_{\Pb^2})^2=\left\{ \begin{array}{l}
         \dfrac{9}{e^2}~~~~~~~~~~~~~~~\text{if}~~~d=3\\
        1-\dfrac{8}{e}+\dfrac{16}{e^2}~~~\text{if}~~~d=4\\
        \dfrac{1}{4}~~~~~~~~~~~~~~~~\text{if}~~~d=5\end{array} \right.
\end{equation}

\subsubsection*{Degree $3$ ramification divisors over $\Pb^2$}

Let $\Xc$ be a maximal order over $\Pb^2$. Also, we let $D$ be ramification divisors of degree $3$ over $\Pb^2$ and we define $\Delta=\left(1-\frac{1}{e}\right)D$. Then $D$ is of one of the types in Figure \ref{Cubic ramification configurations}.  The number $e$  represents the ramification degree of the curves  and $\mu$ is any generator of the cyclic group $\frac{\Zb}{e\Zb}$ and represents the ramification index of the  curves at the branch points.

\begin{figure}[H]
   \centering
\begin{multicols}{4}

\begin{tikzpicture}
\node (1) at (0.7,0) {e};
\node (1) at (0.7,-1) {e};
\node (1) at (1.6,-0.2) {\bmu};
\node (1) at (1.8,-0.7) {-\bmu};
\node (1) at (2.4,0.2) {\bmu};
\node (1) at (2.3,0.7) {-\bmu};
\coordinate (O) at (1,0);
  \coordinate (A) at (3,0);
\coordinate (B) at (1,-1);
  \coordinate (C) at (3,1);
  \draw[] (B)--(C);
  \draw[] (O) to [bend left=90] (A);
  \draw[] (O) to [bend right=90] (A);  
 \end{tikzpicture}\\
 Transverse
 line \& conic
 
\columnbreak

\begin{tikzpicture}[scale=0.8]
\node (1) at (0,0) {e};
\node (1) at (2,0) {e};
\node (1) at (0.9,-1.3) {e};
\node (1) at (-0.5,-0.8) {-\bmu};
\node (1) at (-0.6,-1.3) {\bmu};
\node (1) at (1.5,-0.8) {\bmu};
\node (1) at (2.1,-1.4) {-\bmu};
\node (1) at (0.6,0.75) {\bmu};
\node (1) at (1.4,0.75) {-\bmu};
\draw [black!90]  (-1,-1) -- (3,-1); 
\draw [black] (-0.5,-1.5) -- (1.5,1.3);
\draw [black] (0.5,1.3) -- (2.5,-1.5);
\end{tikzpicture}\\
Transverse lines
\columnbreak

 \begin{tikzpicture}[scale=0.8]
\node (1) at (-1.5,0) {e};
\draw [black, scale=1] plot [smooth, tension=0.8] coordinates { (1.5,1.5) (0.7,0.4)(-0.5,0.8) (-1.2,0) (-0.5,-0.8) (0.7,-0.4)(1.5,-1.5)};
\end{tikzpicture}\\
 A smooth cubic
\columnbreak

\begin{tikzpicture}
\node (1) at (-1.7,0) {e};
\node (1) at (0.2,0.4) {\bmu};
\node (1) at (0.2,-0.4) {-\bmu};
\draw [black, scale=1.2] plot [smooth, tension=1] coordinates { (0.8,0.8) (0,0) (-1,-0.5) (-1,0.5) (0,0) (0.8,-0.8)};
\end{tikzpicture}\\
 A nodal cubic
\end{multicols}
    
\caption{Cubic ramification configurations}
\label{Cubic ramification configurations}
 \end{figure}
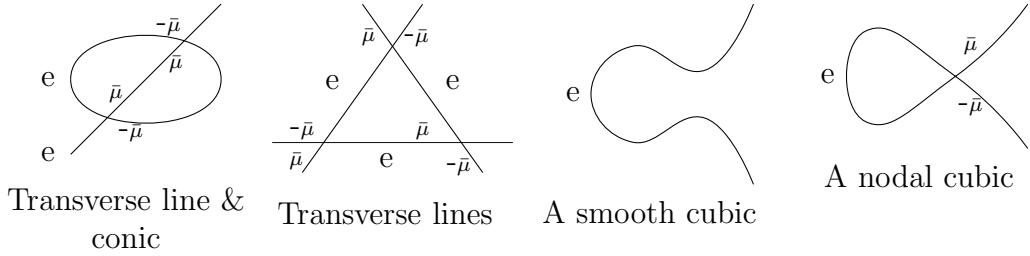

\begin{lemma}\label{blowup out of D, e=2}
Let $\Xc$ be a terminal order over $\Pb^2$ and let $D$ be ramification divisors of degree $3$ with the ramification degree $e$. Also let $f:Z\rightarrow \Pb^2$ be a blowup at a point not in $D$. Then the order over $Z$  is del Pezzo if and only if $e=2$.
\end{lemma}

\begin{proof}
Let $f:Z\rightarrow \Pb^2$ be a blowup at a point $p\notin D$. Then by Proposition \ref{decrease of K_X^2} we have
\[
(K_{Z}+\Delta_Z)^2=\dfrac{9}{e^2}-1.
\]
Therefore $\dfrac{9}{e^2}-1>0$ if and only if $e<3$. 

  Now let $e=2$ and let $C$ be an effective curve in $Z$. Then by 
\begin{align*}
(K_{Z}+\Delta_Z)C&=(f^*(K_{\Pb^2}+\Delta_{\Pb^2})+E)C\\
&=f^*(K_{\Pb^2}+\Delta_{\Pb^2})C+EC\\
&=\left(-3H+\left(1-\dfrac{1}{2}\right)3H\right)f_{*}C+EC\\
&=-\dfrac{3}{2}d+r,
\end{align*}
where $d=\deg(f_*C)$ and $r$ is the multiplicity of $f_*C$ at $p$ which is not greater than $d$. So $(K_{Z}+\Delta_Z)C<0$.
\end{proof}

\begin{proposition}\label{blowup 2 points out of D}
Let $\Xc$ be a TdPO over $\Pb^2$ and let $D$ be ramification divisors  of degree $3$. Also Let $f:Z\rightarrow \Pb^2$ be two blowups at  the points $p$ and $q$ and let $\Yc$ be a maximal order over $Z$ obtained by the blowups. If any of $p$ or $q$ is not in $D$, then $\Yc$ is not del Pezzo.
\end{proposition}

\begin{proof}
By Lemma \ref{blowup out of D, e=2} we know if $e\neq2$, then the order is not del Pezzo. So let $e=2$ and without loss of generality, let us assume $p\notin D$. Then
\[
K_Z+\Delta_Z\equiv f^*(K_{\Pb^2}+\Delta_{\Pb^2})+E_p+aE_q,
\]
where $a\geq\dfrac{1}{2}$ depending on $q$ whether it is in the ramification divisors or not. Now let  $C\in Z$ be a line going through $p$ and $q$. Then
\begin{align*}
(K_{Z}+\Delta_Z)C&=(f^*(K_{\Pb^2}+\Delta_{\Pb^2})+E_p+aE_q)C\\
&=f^*(K_{\Pb^2}+\Delta_{\Pb^2})C+E_pC+aE_qC\\
&=\left(-3H+\left(1-\dfrac{1}{2}\right)3H\right)f_*C+E_pC+aE_qC\\
&=-\dfrac{3}{2}+1+a\geq0,
\end{align*}
 i.e. $\Yc$ is not del Pezzo.
\end{proof}

Now  the only remaining case for degree $3$ ramification divisors to be classified is blowups of points sitting all in $D$.
\begin{definition}
 Let $\Sigma=\{p_1, \cdots, p_n\}$ be a set of distinct points of $\Pb^2$.  $\Sigma$ is in general position if
 \begin{enumerate}
  \item No three points lie on a line;
  \item No six points lie on a conic.
 \end{enumerate}
\end{definition}

\begin{theorem}\label{classification of deg D=3}
Let $\Xc$ be a TdPO over $\Pb^2$ with  ramification divisors $D$ of degree $3$ and  ramification degrees $e$. Let $f:Z\rightarrow \Pb^2$ be a sequence of blowups at the points $\Sigma=\{p_1, \cdots, p_n\}$. Then the associated maximal order $\Yc$ over $Z$ is del Pezzo if and only if one of the following occurs.
\begin{enumerate}
\item 
$\Sigma\subset D$, $\Sigma$ is in general position and $n<9$;
 \item 
 $\Sigma\not\subset D$, $n=1$ and $e=2$.
\end{enumerate}
\end{theorem}

\begin{proof}
Here the ramification divisor is anti-canonical  and therefore the first case is proved by Proposition \ref{Z iff X} and classification of del Pezzo blowups of $\Pb^2$ given in \cite{hidaka1981normal}.  The second case is Proposition \ref{blowup 2 points out of D}.
\end{proof}

\subsubsection*{Degree $4$ and $5$ Ramification Divisors}
\begin{theorem}\label{blowups degree 4}
 Let $\Xc$ be a TdPO over $\Pb^2$ with ramification divisors $D=\cup D_i$  of degree $4$. Also, let  $f:Z\rightarrow \Pb^2$ be a blowup at a point $p$ and let $\Yc$ be the associated maximal order over $Z$. Then $\Yc$ is del Pezzo if and only if $p\in D$ and $e=2$. Moreover, $\Yc$ can not be blown up to a del Pezzo order. 
\end{theorem}

\begin{proof}
Let $\Xc$ be a TdPO over $\Pb^2$ with the ramification divisor $D=\cup D_i$  of degree $4$. Let  $f:Z\rightarrow \Pb^2$ be a blowup at a point $p$ with the exceptional curve $E$ and let $\Yc$ be the associated maximal order over $Z$. Using Equation \ref{self intersect P^2} and Proposition \ref{decrease of K_X^2} we get
\begin{align*}
(K_{Z}+\Delta_Z)^2&=1-\dfrac{8}{e}+\dfrac{16}{e^2}-a^2,
\end{align*}
where $e=2$ or $e=3$. We can easily get that $a\neq 1$, meaning that $p\in D$ by Lemma \ref{coefficient of E} and also $e\neq3$. We see that $D\sim -K_{\Pb^2}+H$ for a line $H$. Further $p\in D$, and by Proposition \ref{Y dP M} the  sufficient and necessary condition for $\Yc$ to be del Pezzo is to have 
$mult_{p}f_*C<2-(2-1)H.f_*C+f_*C^2$,  for  every effective curve $C\in Z$. Moreover, by Remark \ref{effective cone} we need to check the inequality only for $C=\widetilde{H}$ and $C=E$ which generate the effective cone of $Z$.
\begin{align*}
0=mult_{p}H&<2-H^2+H^2=2
\end{align*}
and
\begin{align*}
0=mult_{p}p&<2-H.p+p^2=1
\end{align*}

Now let $f:Z\rightarrow \Pb^2$ be blowups at the points $p$ and $q$ in $D$ and let $\ell$ be a  line passing through $p$ and $q$. Also let $\tilde{\ell}$ be the proper transform of $\ell$. Then
\begin{align*}
(K_{Z}+\Delta_Z)\tilde{\ell}&=\left(f^*(K_{\Pb^2}+\Delta_{\Pb^2})+\dfrac{1}{2}E_p+\dfrac{1}{2}E_q\right)\tilde{\ell}\\
&=(K_{\Pb^2}+\Delta_{\Pb^2})\ell+\dfrac{1}{2}(E_p+E_q)\ell\\
&\leq -1+1,
\end{align*}
i.e. the obtained  order over $Z$ is not del Pezzo. 
\end{proof}

\begin{theorem}\label{blowups degree 5}
 Let $\Xc$ be a TdPO over $\Pb^2$ with  ramification divisors $D=\cup D_i$  of degree $5$. Then any blowup of $\Xc$ will not be del Pezzo.
\end{theorem}
\begin{proof}
Let $\Xc$ be a TdPO over $\Pb^2$ with the ramification divisor $D=\cup D_i$  of degree $5$. And let  $f:Z\rightarrow \Pb^2$ be a blowup at a point $p$ and let $\Yc$ be the order over $Z$. Using Equation \ref{self intersect P^2} and Proposition \ref{decrease of K_X^2} we get
\begin{align*}
(K_{Z}+\Delta_Z)^2&=\dfrac{1}{4}-a^2,
\end{align*}
where $a=1$ or $a=\frac{1}{2}$, both of which make the self-intersection $(K_{Z}+\Delta_Z)^2$ less than or equal to zero.
\end{proof}

Figures \ref{Quartic ramification divisors}  is a complete list of degree $4$ terminal ramification divisors over $\Pb^2$. The four ones on the top have $e=2$ and the rest have $e=3$.

\begin{figure}[]
  \centering
\begin{multicols}{4}
 \begin{tikzpicture}[scale=0.9]
 \node (1) at (0,1) {};
 \node (1) at (0,-0.3) {\ti1};
\node (1) at (0,-0.9) {\ti1};
\draw [black] plot [smooth, tension=0.9] coordinates { (0,-0.6) (1,0) (1.4,-0.6) (1,-1.2) (0,-0.6) (-1,0) (-1.4,-0.6) (-1,-1.2) (0,-0.6)};
\end{tikzpicture}\\
Irreducible quartic with one node

\columnbreak
\begin{tikzpicture}[scale=0.9]
 \node (1) at (0,1.6) {};
 \node (1) at (0,0.3) {\ti1};
\node (1) at (0,-0.3) {\ti1};
\node (1) at (1.6,0.3) {\ti1};
\node (1) at (1.6,-0.3) {\ti1};
\draw [black] plot [smooth, tension=0.7, scale=0.8] coordinates { (0,0) (1,0.8) (2,0) (3,-0.8) (3.4,0) (3,0.8) (2,0) (1,-0.8) (0,0) (-1,0.8) (-1.4,0) (-1,-0.8) (0,0)};
\end{tikzpicture} \\
Irreducible quartic with two nodes

\columnbreak
\begin{tikzpicture}
\node (1) at (-0.1,0.6) {\ti1};
\node (1) at (-0.5,0.7) {\ti1};
\node (1) at (0.5,0.2) {\ti1};
\node (1) at (-0.2,-0.2) {\ti1};
\node (1) at (-0.5,-0.3) {\ti1};
\node (1) at (0.5,-0.2) {\ti1};
\draw [black] plot [smooth cycle, tension=2, scale=1.2] coordinates { (-0.3,0) (0.8,1.1) (0.1,-0.2) (-1.4,0) (0,0.2) (0.9,-1)};
 \end{tikzpicture} \\
Irreducible quartic with three nodes
 
\columnbreak
 \begin{tikzpicture}[scale=0.8]
 \node (1) at (0,1.8) {};
\node (1) at (-0.6,0.6) {\ti1};
\node (1) at (-1,1) {\ti1};
\node (1) at (0.6,0.6) {\ti1};
\node (1) at (1,1) {\ti1};
\node (1) at (-0.6,-0.6) {\ti1};
\node (1) at (-1,-1) {\ti1};
\node (1) at (0.6,-0.6) {\ti1};
\node (1) at (1,-1) {\ti1};
\draw [black] plot [smooth cycle, tension=1] coordinates { (-1,0) (0,1.5) (1,0) (0,-1.5)};
\draw [black] plot [smooth cycle, tension=1] coordinates { (-1.5,0) (0,1) (1.5,0) (0,-1)};
\end{tikzpicture}\\
Two conics crossing at $4$ points.
\end{multicols} 
\begin{multicols}{3}
 \begin{tikzpicture}[scale=0.9]
 \node (1) at (-1.8,1.6) {};
 \node (1) at (0,0.3) {\ti1};
\node (1) at (0,-0.3) {\ti2};
\draw [black] plot [smooth, tension=0.9] coordinates { (0,0) (1,0.6) (1.4,0) (1,-0.6) (0,0) (-1,0.6) (-1.4,0) (-1,-0.6) (0,0)};
\end{tikzpicture}\\
Irreducible quartic with one node

 \begin{tikzpicture}[scale=1.1]
 \node (1) at (0,1.6) {};
\node (1) at (-1.2,0) {\ti1};
\node (1) at (-1.5,0.4) {\ti2};
\node (1) at (-0.6,0.1) {\ti1};
\node (1) at (-0.5,0.6) {\ti2};
\node (1) at (0.7,0.3) {\ti1};
\node (1) at (0.6,0.7) {\ti2};
\node (1) at (0.3,0) {\ti1};
\node (1) at (0,-0.2) {\ti2};
\draw [black, scale=1.2] plot [smooth, tension=1] coordinates { (0.8,0.8) (0,0) (-1,-0.5) (-1,0.5) (0,0) (0.8,-0.8)};
\draw [black, scale=1.2] plot [smooth, tension=1] coordinates { (-1.6,0.1) (0.9,0.4)};
\end{tikzpicture}\\
Transverse\\nodal cubic \& line

\begin{tikzpicture}
\node (1) at (-0.4,1.7) {};
\node (1) at (-1.1,-0.5) {\ti2};
\node (1) at (-0.6,-0.7) {\ti1};
\node (1) at (1.1,0.6) {\ti1};
\node (1) at (0.8,0.9) {\ti2};
\node (1) at (0.4,0.6) {\ti2};
\node (1) at (0.6,0.15) {\ti1};
\draw [black, scale=1] plot [smooth, tension=1] coordinates { (-1.5,-1) (1.6,1.3)};
\draw [black, scale=1] plot [smooth, scale=0.8, tension=0.8] coordinates { (1.5,1.5) (0.7,0.4)(-0.5,0.8) (-1.2,0) (-0.5,-0.8) (0.7,-0.4)(1.5,-1.5)};
\end{tikzpicture}\\
Transverse\\ smooth cubic \& line
 
\columnbreak
  \begin{tikzpicture}[scale=0.9]
 \node (1) at (0,1.6) {};
 \node (1) at (0,0.3) {\ti2};
\node (1) at (0,-0.3) {\ti1};
\node (1) at (1.6,0.3) {\ti2};
\node (1) at (1.6,-0.3) {\ti1};
\draw [black] plot [smooth, tension=0.7, scale=0.8] coordinates { (0,0) (1,0.8) (2,0) (3,-0.8) (3.4,0) (3,0.8) (2,0) (1,-0.8) (0,0) (-1,0.8) (-1.4,0) (-1,-0.8) (0,0)};
\end{tikzpicture}\\
Irreducible quartic with two node
 
\begin{tikzpicture}[scale=0.8]
\node (1) at (-1.7,1.8) {};
\node (1) at (-0.6,0.7) {\ti1};
\node (1) at (-0.9,1) {\ti2};
\node (1) at (0.6,0.7) {\ti1};
\node (1) at (0.9,1) {\ti2};
\node (1) at (-0.6,-0.7) {\ti2};
\node (1) at (-0.9,-1.1) {\ti1};
\node (1) at (0.6,-0.7) {\ti2};
\node (1) at (0.9,-1.1) {\ti1};
\draw [black] plot [smooth cycle, tension=1] coordinates { (-1,0) (0,1.5) (1,0) (0,-1.5)};
\draw [black] plot [smooth cycle, tension=1] coordinates { (-1.5,0) (0,1) (1.5,0) (0,-1)};
\end{tikzpicture}\\
Transverse  conics

 \begin{tikzpicture}[scale=1.1]
 \node (1) at (0,1.6) {};
\node (1) at (-1.2,0) {\ti2};
\node (1) at (-1.5,0.4) {\ti1};
\node (1) at (-0.6,0.1) {\ti2};
\node (1) at (-0.5,0.6) {\ti1};
\node (1) at (0.7,0.3) {\ti2};
\node (1) at (0.6,0.7) {\ti1};
\node (1) at (0.3,0) {\ti1};
\node (1) at (0,-0.2) {\ti2};
\draw [black, scale=1.2] plot [smooth, tension=1] coordinates { (0.8,0.8) (0,0) (-1,-0.5) (-1,0.5) (0,0) (0.8,-0.8)};
\draw [black, scale=1.2] plot [smooth, tension=1] coordinates { (-1.6,0.1) (0.9,0.4)};
\end{tikzpicture}
Transverse\\ nodal cubic \& line
 
\columnbreak
 \begin{tikzpicture}
\node (1) at (-0.1,0.6) {\ti1};
\node (1) at (-0.5,0.7) {\ti2};
\node (1) at (0.5,0.2) {\ti1};
\node (1) at (-0.2,-0.2) {\ti2};
\node (1) at (-0.5,-0.3) {\ti1};
\node (1) at (0.5,-0.2) {\ti2};
\draw [black] plot [smooth cycle, tension=2, scale=1.2] coordinates { (-0.3,0) (0.8,1.1) (0.1,-0.2) (-1.4,0) (0,0.2) (0.9,-1)};
 \end{tikzpicture} \\
 Irreducible quartic with three node

\begin{tikzpicture}
\node (1) at (-0.4,1.7) {};
\node (1) at (-1.1,-0.5) {\ti1};
\node (1) at (-0.6,-0.7) {\ti2};
\node (1) at (1.1,0.6) {\ti2};
\node (1) at (0.8,0.9) {\ti1};
\node (1) at (0.4,0.6) {\ti1};
\node (1) at (0.6,0.15) {\ti2};
\draw [black, scale=1] plot [smooth, tension=1] coordinates { (-1.5,-1) (1.6,1.3)};
\draw [black, scale=1] plot [smooth, scale=0.8, tension=0.8] coordinates { (1.5,1.5) (0.7,0.4)(-0.5,0.8) (-1.2,0) (-0.5,-0.8) (0.7,-0.4)(1.5,-1.5)};
\end{tikzpicture}\\
Transverse\\ smooth cubic \& line

\begin{tikzpicture}
\node (1) at (-0.4,0.5) {\ti2};
\node (1) at (-0.4,1.2) {};
\node (1) at (-0.8,0.75) {\ti1};
\node (1) at (-0.5,-0.4) {\ti1};
\node (1) at (-0.8,-0.8) {\ti2};
\node (1) at (1,0.4) {\ti2};
\node (1) at (0.6,0.7) {\ti1};
\node (1) at (0.9,-0.2) {\ti2};
\node (1) at (0.8,-0.7) {\ti1};
\node (1) at (0.05,0.15) {\ti1};
\node (1) at (-0.1,-0.25) {\ti2};
\draw [black, scale=0.8] plot [smooth cycle, tension=1] coordinates { (-1,0) (0,1) (1,0) (0,-1)};
\draw [black, scale=1] plot [smooth, tension=1] coordinates { (-1,-1.2) (1.2,0.8)};
\draw [black, scale=1] plot [smooth, tension=1] coordinates { (-1.4,1) (1.2,-0.8)};
\end{tikzpicture} 
Transverse\\  conic \&  lines

\end{multicols}
\caption{Quartic ramification divisors}
\label{Quartic ramification divisors}
\end{figure}
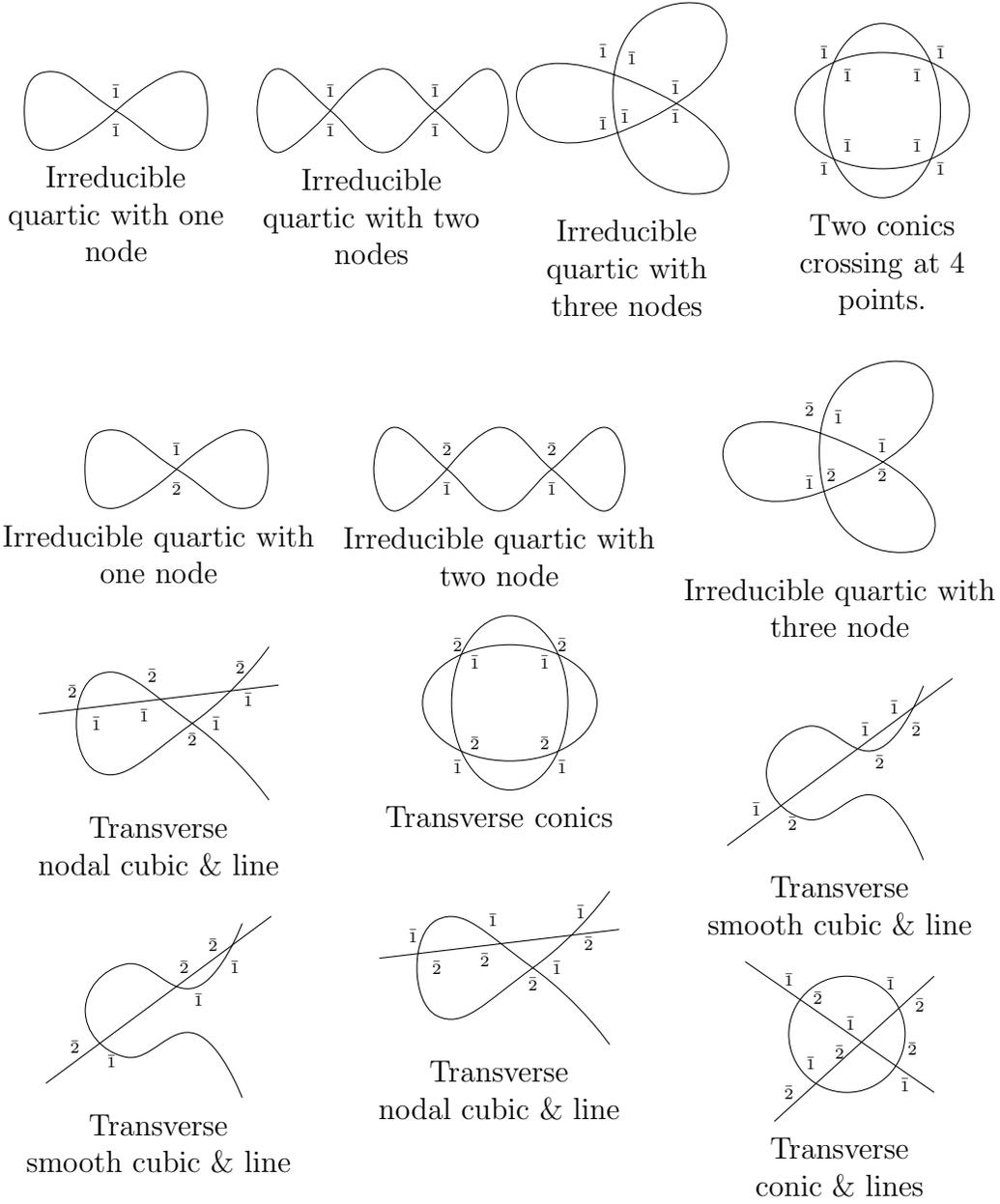

\subsection{TdPOs over $\PP$}
\begin{proposition}\cite[Proposition 30]{chan2003pezzo}\label{centre of del pezzo on PP}
If $D$ is the ramification divisors  of a TdPO over $\PP$, then $2C_0+2F\leq[D]\leq 3C_0+3F$ where $C_0$ and $F$ are two crossing fibres . Moreover the ramification degrees of $D$ are equal, say $e$, and  $e=2$ if $2C_0+2F< [D]\leq 3C_0+3F$. Conversely any terminal order with such ramification data is del Pezzo.
\end{proposition}
Let $\Xc$ be a TdPO over $\PP$ and let $D$ be its ramification divisors with  ramification degree $e$, then the canonical divisor is as the following
\begin{align*}
K_{\Xc}=K_{\PP}+\Delta_{\PP}&\equiv \left(-2C_0-2F\right)+\left(1-\frac{1}{e}\right)\left(aC_0+bF\right)\\
&=\left(a-2-\dfrac{a}{e}\right)C_0+\left(b-2-\dfrac{b}{e}\right)F.
\end{align*}
for suitable $a$ and $b$. Further, we have the following equalities  for the self-intersection of the canonical divisor
\[
(K_{\PP}+\Delta_{\PP})^2=2\left(a-2-\dfrac{a}{e}\right)\left(b-2-\dfrac{b}{e}\right).
\]
Thus
\begin{equation}\label{self intersect P^1}
(K_{\PP}+\Delta_{\PP})^2=\left\{ \begin{array}{l}
         \dfrac{8}{e^2}~~~~~~~~~~~~~~~\text{if}~~~D\sim 2C_0+2F\\
        1~~~\text{if}~~~D\sim 3C_0+2F ~\text{or}~D\sim 2C_0+3F\\
        \dfrac{1}{2}~~~~~~~~~~~~~~~~\text{if}~~~D\sim 3C_0+3F\end{array} \right.
\end{equation}

\begin{proposition}\label{blowup p notin PP}
 Let $\Xc$ be a TdPO over $\PP$ and let $D$ be its ramification divisors. Then if $f:Z\rightarrow\PP$ is a blowup at a point $p\notin D$, then the order $\Yc$ over $Z$ is not del Pezzo.
\end{proposition}

\begin{proof}
Let $f:Z\rightarrow\PP$ represents a blowup at a point $p\notin D$. Then if $E$ is the corresponding exceptional curve, we have the following equations
\begin{align*}
K_{Z}+\Delta_Z&\equiv f^*(K_{\PP}+\Delta_{\PP})+E,\\
&=f^*\left(\left(a-2-\dfrac{a}{e}\right)C_0+\left(b-2-\dfrac{b}{e}\right)F\right)+E.
\end{align*}

\begin{align*}
(K_{Z}+\Delta_Z)^2&=2\left(a-2-\dfrac{a}{e}\right)\left(b-2-\dfrac{b}{e}\right)-1
\end{align*}
It can be easily checked that the self-intersection in all cases of $D$ except $D\equiv 2C_0+2F$ is negative and  for $D\equiv 2C_0+2F$ it is positive only if the ramification degrees are $2$. Now let $D\equiv 2C_0+2F$ and let $\tilde{\ell}$ be the proper transform of a line $\ell\equiv F$ going through $p$. Then
\begin{align*}
(K_{Z}+\Delta_Z)\tilde{\ell}&= (f^*(K_{\PP}+\Delta_{\PP}))\tilde{\ell}+E.\tilde{\ell}\\
&=(-C_0-F)\ell+1\\
&=0
\end{align*}
\end{proof}

\begin{theorem}\label{blowup p in PP}
Let $\Xc$ be a TdPO over $\PP$ and let $D$ be its ramification divisors such that $2C_0+2F<[D]$. Then any blowup of the order at a point in $D$ fails to be del Pezzo.
\end{theorem}

\begin{proof}
Let $f:Z\rightarrow\PP$ represent a blowup at a point $p\in D$ with the exceptional curve $E$ where $[D\sim2C_0+2F+M]\leq 3C_0+3F$ for some effective divisor $M$. Then we know the ramification degrees $e=2$ and by Proposition  \ref{Y dP M} the inequality $mult_{p}f_*C<2-(2-1)H.f_*C+f_*C^2$ should satisfy for any effective curve $C\in Z$. But by Remark \ref{effective cone} we only need to check for the only $(-1)$-curve, $E$.
\begin{align*}
0=mult_{p}p<1-M.p=1
\end{align*}
\end{proof}

\begin{definition}
Let $\Sigma$ be a set of distinct points in $\PP$. $\Sigma$ is in general position if any curve of the form $aC_0+bF$ contains less than $2(a+b)$  points of $\Sigma$.
\end{definition}
\begin{theorem}\label{blowup p in D PP}
Let $\Xc$ be a TdPO over $\PP$ and let $D$ be its ramification divisors such that $[D]=2C_0+2F$. Also let $f:Z\rightarrow\PP$ be a sequence of blowups at the points $\Sigma=\{p_1, \cdots, p_n\}\subset D$. Then the associated order $\Yc$ over $Z$ is del Pezzo if and only if $\Sigma$ is in general position and $n\leq 7$.
\end{theorem}

\begin{proof}
Here $D$ is an anti-canonical ramification, then use Proposition \ref{Z iff X} and classification of del Pezzo blowups of $\PP$ given in \cite{hidaka1981normal}.
\end{proof}

Figures  \ref{D=22} and \ref{figure e=2 on P1P1} are  the lists of possible minimal TdPO over $\PP$. In the descriptions, by $(a,b)$ it means a curve  in the divisor class $[aC_0+bF]$.

\begin{figure}[H]
  \centering
\begin{multicols}{3}
 \begin{tikzpicture}[scale=1]
 \node (1) at (0,1) {};
 \node (1) at (1,1) {\s$e$};
 \node (1) at (0,1) {\s$e$};
 \node (1) at (-0.8,0) {\s$e$};
 \node (1) at (-0.8,-1) {\s$e$};
 \node (1) at (1.2,0.3) {\bmu};
 \node (1) at (0.8,0.1) {-\bmu};
 \node (1) at (-0.2,-0.2) {-\bmu};
 \node (1) at (-0.3,0.15) {\bmu};
 
 \node (1) at (1.3,-1.1) {\bmu};
 \node (1) at (0.8,-0.8) {-\bmu};
 \node (1) at (-0.3,-1.2) {-\bmu};
 \node (1) at (-0.1,-0.8) {\bmu};
\draw [black] plot coordinates {(-0.6,0) (1.6,0)};
\draw [black] plot coordinates {(-0.6,-1) (1.6,-1)};
\draw [black] plot coordinates {(0,0.6) (0,-1.6)};
\draw [black] plot coordinates {(1,-1.6) (1,0.6)};
\end{tikzpicture}
$2\times(1,0)+2\times(0,1)$

\columnbreak
 
 \begin{tikzpicture}[scale=0.9]
 \node (1) at (0,1) {};
 \node (1) at (0,1) {\s$e$};
 \node (1) at (-0.6,1) {\s$e$};
 \node (1) at (-0.8,-1) {\s$e$};
 \node (1) at (-0.3,0) {-\bmu};
 \node (1) at (0.15,0.25) {\bmu};
 \node (1) at (0.2,-0.7) {-\bmu};
 \node (1) at (-0.2,-1.2) {\bmu};
 \node (1) at (1.25,-0.8) {-\bmu};
 \node (1) at (1,-1.25) {\bmu};
 \draw [black] plot coordinates {(-0.6,0.6) (1.6,-1.6)};
\draw [black] plot coordinates {(-0.6,-1) (1.6,-1)};
\draw [black] plot coordinates {(0,0.6) (0,-1.6)};
\end{tikzpicture}
$(1,0)+(0,1)+(1,1)$
\columnbreak
  
 \begin{tikzpicture}[scale=0.9]
 \node (1) at (0,1) {};
 \node (1) at (-0.3,1) {\s$e$};
 \node (1) at (-0.8,0) {\s$e$};
 \node (1) at (0.15,0.2) {-\bmu};
 \node (1) at (-0.25,-0.3) {\bmu};
 \node (1) at (1.45,0.3) {-\bmu};
 \node (1) at (0.9,-0.15) {\bmu};
\draw [black] plot coordinates {(-0.6,0) (1.6,0)};
 \draw [black] plot [smooth, tension=0.6] coordinates {(-0.3,0.6) (0,-0.6) (0.5,-1.2) (1,-0.6) (1.3,0.6)};
 \end{tikzpicture}
 $(1,0)+(1,2)$
\end{multicols}
 \begin{multicols}{5}
~ 
\columnbreak
 
\begin{tikzpicture}[scale=0.9]
 \node (1) at (0,1.6) {};
 \node (1) at (1,1) {\s$e$};
 \node (1) at (-0.1,0.3) {-\bmu};
 \node (1) at (-0.1,-0.3) {\bmu};
\draw [black] plot [smooth, tension=0.9] coordinates { (0,0) (1,0.6) (1.4,0) (1,-0.6) (0,0) (-1,0.6) (-1.4,0) (-1,-0.6) (0,0)};
\end{tikzpicture}
 $(2,2)$

\columnbreak
 ~
 \columnbreak

 \begin{tikzpicture}[scale=0.9]
 \node (1) at (0,1.6) {};
 \node (1) at (1,1) {\s$e$};
 \node (1) at (-0.1,0.3) {-\bmu};
 \node (1) at (-0.1,-0.3) {\bmu};
 \node (1) at (1.7,0.3) {-\bmu};
 \node (1) at (1.7,-0.3) {\bmu};
 \draw [black] plot [smooth, tension=0.7, scale=0.8] coordinates { (0,0) (1,0.8) (2,0) (3,-0.8) (3.4,0) (3,0.8) (2,0) (1,-0.8) (0,0) (-1,0.8) (-1.4,0) (-1,-0.8) (0,0)};
\end{tikzpicture}
 $(2,2)$

\columnbreak
 ~
 
\end{multicols}

\caption{$[D]=2C_0+2F$}\label{D=22}
\end{figure}
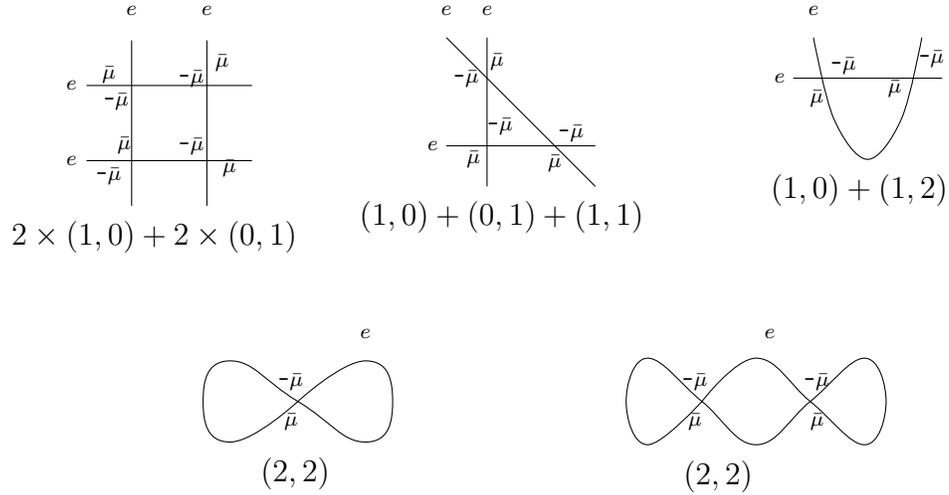

\begin{figure}[H]
  \centering
\begin{multicols}{4}
 \begin{tikzpicture}[scale=0.9]
 \node (1) at (0,1) {};
\draw [black] plot coordinates {(-0.6,0) (1.6,0)};
\draw [black] plot coordinates {(-0.6,-0.8) (1.6,-0.8)};
 \draw [black] plot [smooth, tension=0.6] coordinates {(-0.3,0.6) (0,-0.6) (0.5,-1.2) (1,-0.6) (1.3,0.6)};
 \end{tikzpicture}
 $2\times(1,0)+(1,2)$
   
\columnbreak

 \begin{tikzpicture}[scale=0.9]
 \node (1) at (1.7,1.2) {};
\draw [black] plot [smooth, tension=0.9] coordinates { (0,0) (1,0.6) (1.4,0) (1,-0.6) (0,0) (-1,0.6) (-1.4,0) (-1,-0.6) (0,0)};
\draw [black] plot coordinates {(1,1) (1,-1)};
\end{tikzpicture}
 $(2,2)+(0,1)$

\columnbreak
 
 \begin{tikzpicture}[scale=0.9]
 \node (1) at (3.6,1.4) {};
\draw [black] plot [smooth, tension=0.7, scale=0.8] coordinates { (0,0) (1,0.8) (2,0) (3,-0.8) (3.4,0) (3,0.8) (2,0) (1,-0.8) (0,0) (-1,0.8) (-1.4,0) (-1,-0.8) (0,0)};
\draw [black] plot coordinates {(0.5,1.2) (0.5,-1.2)};
\end{tikzpicture}
 $(2,2)+(0,1)$

\columnbreak
 
 \begin{tikzpicture}[scale=1.1, rotate=40]
 \node (1) at (0,1) {};
\draw [black] plot coordinates {(-0.6,0) (2,0)};
 \draw [black] plot [smooth, tension=0.6] coordinates {(-0.5,0.4) (-0.1,-0.1) (0.5,-0.3) (1.1,-0.1) (1.5,0.4)};
 \draw [black] plot [smooth, tension=0.6] coordinates {(-0.1,-0.5) (0.3,0) (0.9,0.3) (1.5,0) (1.9,-0.5)};
 \end{tikzpicture}
 $3\times (1,1)$
 
\end{multicols}

\caption{$2C_0+2F<[D]\leq3C_0+3F, e=2$}\label{figure e=2 on P1P1}
\end{figure}
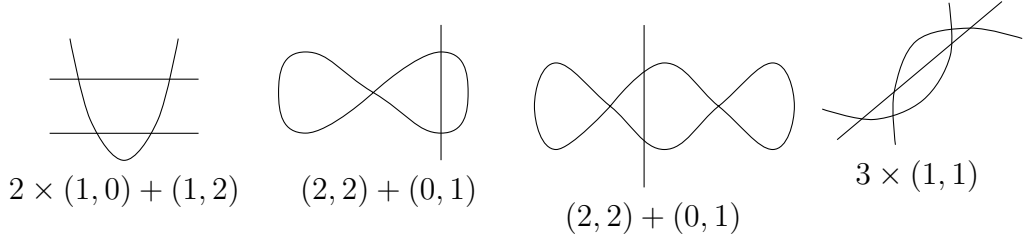

\section{Minimal Models of CdPOs}\label{Minimal Models of CdPOs}

In this section we  classify minimal models of canonical del Pezzo orders (CdPOs). This classification will be used in Section \ref{Classification of CdPOs} to classify all CdPOs. 

\subsection{The Terminal resolution of CdPOs}\label{section Minimal Terminal orders for del Pezzo Canonical Orders}
\begin{definition}\label{minimal resolution for canonical orders}
Let $\Xc$ be a canonical order over $Z$. Then by \cite[Corollary 3.6]{CI2} there exists a sequence of blowups $f:\Yc\rightarrow \Xc$, called resolution of $\Xc$, where $\Yc$ is a terminal order. The resolution $f:\Yc\rightarrow \Xc$ is called minimal if $K_{\Yc}.E\geq 0$ for every $f$-exceptional curve $E$.
\end{definition}

\begin{definition}
Let  $\Xc$ be an order and let $f:\Yc\rightarrow\Xc$ be a resolution of $\Xc$. Then we have the following equation
\[
K_{\Yc}=K_{\Xc}+\sum_ia_iE_i,
\]
where  $E_i$s are $f$-exceptional curves. The order $\Xc$ has canonical singularities if $a=\min\{a_i\}\geq 0$.
\end{definition}

\begin{definition}\label{almost del pezzo order}
Let $\Xc$ be a maximal order over a Gorenstein surface $Z$. We call $\Xc$ almost del Pezzo if $ K_{\Xc}^2>0$, and for every effective curve $C\in Z$, $-K_{\Xc}C\geq 0$. By \cite[Theorem 2.2.16.]{lazarsfeld2017positivity} it is equivalent to $-K_{\Xc}$ being nef and big.
\end{definition}

\begin{proposition}
Let $\Xc$ be an order over a normal Gorenstein surface $Z$. Then $\Xc$ is almost del Pezzo if and only if $-K_{\Xc}$ is big and semiample.
\end{proposition}

\begin{proof}
Let $\Xc$ be an order over a normal Gorenstein surface $Z$ and let $-K_{\Xc}$ be  semiample. Thus restriction of $-K_{\Xc}$ to any curve $C$, $-{K_{\Xc}}_{|_C}$, is semiample and therefore $-K_{\Xc}C\geq 0$. So $-K_{\Xc}$ is  nef. Now let $-K_{\Xc}$ be nef and big. We show that $K_Z.K_{\Xc}>0$ then by \cite[Lemma 3.1]{laface2011nef} we see that $-K_{\Xc}$ is semiample. As $-K_{\Xc}$ is big and nef and $\Delta$ is effective, then $-K_{\Xc}(-K_Z-\Delta)=(-K_{\Xc})^2>0$ and $-K_{\Xc}\Delta\geq 0$. So we get $K_Z.K_{\Xc}>0$.
\end{proof}

\begin{proposition}\cite[Proposition 6.1]{chan2008canonical}\label{canonical bundle of resolution} 
Let $\Xc$ be a canonical order and let $f:\Yc\rightarrow\Xc$ be its minimal resolution. Then $K_{\Yc}=f^*K_{\Xc}$.
\end{proposition}

\begin{corollary}\label{minimal resolution of canonical orders}
 Let $\Xc$ be a del Pezzo canonical order and let $f:\Yc\rightarrow\Xc$ be its minimal resolution. Then $\Yc$ is almost del Pezzo.
\end{corollary}

\begin{corollary}\label{MMP for almost terminal}
Let $\Yc$ be a terminal almost del Pezzo order over $Z_{\Yc}$ and  let $\Wc$ over $Z_{\Wc}$ be a minimal terminal order obtained by running the minimal model program on $\Yc$. Then $\Wc$ is almost del Pezzo.
\end{corollary}

\begin{lemma}\cite[Theorem 1]{chan2003pezzo}\label{C is smooth rational}
 Let $Z$ be a surface with canonical singularities and let $C$ be an irreducible curve. If $(K_Z+C)C<0$ and $K_ZC\geq0$, then $C$ is a smooth rational curve.
\end{lemma}

Chan and Kulkarni showed that if an order $\Xc$ on a normal Gorenstein surface $Z$ is del Pezzo, then the centre is del Pezzo. We want to generalize their result to almost del Pezzo orders. The proof is mostly the same, however, we need to prove the following Lemma.

\begin{lemma}\label{Effective curve.K_X}
Let $\Xc$ be a maximal order over a normal Gorenstein surface $Z$. If $\Xc$ is almost del Pezzo, then for every irreducible curve $C$, $K_ZC\leq 0$.
\end{lemma}
\begin{proof}

 Chan and Kulkarni showed that if an order $\Xc$ on the centre $Z$ is del Pezzo, then $K_ZC<0$ for every irreducible curve $C\in Z$, \cite[Theorem 12]{chan2003pezzo}. To do so, by contradiction it is assumed  that there is an irreducible curve $C$, such that $K_{\Xc}C<0$ but $K_ZC\geq 0$. Then the curve $C$ is a smooth rational curve, and it leads to a contradiction. Here we only need to show that if for any curve $C$, $K_{\Xc}C\leq 0$ and  $K_ZC>0$, then $C$ is smooth rational. For then,  the same contradiction would be reached.
 
 Thus let $\Delta=\sum_i(1-\frac{1}{e_i})D_i$ be the ramification configuration for the order $\Xc$ and let $C$ be an irreducible curve in $Z$.  If $C$ is not one of the ramification divisors $D_i$, then $\Delta C\geq0$ and the arguments is proved by the following equation
\[
 K_XC=(K_Z+\Delta)C\leq0.
\]
If $C$ is a ramification divisor, then without loss of generality we can assume that $C=D_1$ and the  ramification degree for $C$, $e=e_1$. $\Xc$ is almost del Pezzo, so $K_{\Xc}C\leq0$.
\begin{align*}
 0\geq K_{\Xc}C&=\left(K_Z+\sum_i\left(1-\frac{1}{e_i}\right)D_i\right)C\\
 &=\left(\frac{1}{e}(K_ZC)+\left(1-\frac{1}{e}\right)(K_Z+C)C+\sum_{i\neq1}\left(1-\frac{1}{e_i}\right)D_i\right)C.
\end{align*}
By contradiction, let $K_ZC>0$. As $\left(\sum_{i\neq1}\left(1-\frac{1}{e_i}\right)D_i\right)C\geq0$, then $(K_Z+C)C<0$. So by Lemma \ref{C is smooth rational},  we conclude that $C$ is a smooth rational curve.
\end{proof}

\begin{theorem}\label{almost del pezzo centre} 
Let $\Xc$ be a maximal order on a normal Gorenstein surface $Z$. Then if $\Xc$ is almost del Pezzo, so is $Z$.
\end{theorem}
\begin{proof}
$\Xc$ is almost del Pezzo, so $K_{\Xc}\Delta\leq0$ and $0<K_{\Xc}^2$, also by Lemma \ref{Effective curve.K_X}, $K_Z\Delta\leq0$. Then
\begin{align*}
0<K_{\Xc}^2&=K_{\Xc}(K_Z+\Delta)\\
&=K_{\Xc}K_Z+K_{\Xc}\Delta\\
&\leq K_{\Xc}K_Z\\
&=(K_Z+\Delta)K_Z\\
&=K_Z^2+K_Z\Delta\\
&\leq K_Z^2.
\end{align*}
This together with Lemma \ref{Effective curve.K_X} finishes the proof.
\end{proof}

\begin{theorem}\label{minimal terminal almost del Pezzo order}
Let  $\Wc$ be a  minimal terminal almost del Pezzo order on $Z$. Then we have one of the followings 

\begin{enumerate}
\item
$Z=\Pb^2$, and $\Wc$ is del Pezzo;
\item
$Z$  is a rational ruled surface. More precisely, $Z=\Fb_n$ for $n=0,1,$ or $2$.
\end{enumerate}
\end{theorem}

\begin{proof}
 By \cite[Corollary 3.20]{CI2} we know for minimal  terminal orders, we have one of the following

 \begin{enumerate}
\item
$Z=\Pb^2$, and $\Wc$ is del Pezzo;
\item
$Z\rightarrow C$  is a ruled surface for a smooth rationa curve $C$.
\end{enumerate}
So we only need to show if for  a minimal terminal almost del Pezzo order the later occurs, then the surface is rationally ruled and  it is $\Fb_0, \Fb_1$, or $\Fb_2$.

Let $\Wc$ be a minimal almost del Pezzo order on $Z$ and let $\pi:Z\rightarrow C$ be the morphism surjecting $Z$ to the curve $C$. We note the arithmetic genus of $C$ by $g$. By the genus formula for ruled surfaces we have 
\begin{align*}
 K_{Z}C_0&=n+2g-2~\text{and}\\
 K_{Z}^2&=8(1-g),
\end{align*}
where $n=-C_0^2$. By Theorem \ref{almost del pezzo centre} we have that $Z$ is almost del Pezzo, so $K_{Z}^2>0$ and $K_{Z}C_0\leq 0$. 
Therefore $g=0$ and $n\leq 2$, i.e.
$Z_{\Wc}=\Fb_n$ for $0\leq n\leq 2$.
\end{proof}

\subsection{Minimal TAdPOs over Ruled Surfaces}\label{section Minimal almost del Pezzo terminal orders on rational ruled surfaces}
In this section we classify minimal TAdPOs over ruled surfaces $\Pb^1\times\Pb^1, \Fb_1$ and $\Fb_2$ . We let $Z$ denote any of these ruled surfaces if it is not specified which. 

\begin{proposition}\cite[Lemma 23]{chan2003pezzo}\label{genus for p-max}
Let $p$ be a prime integer dividing some ramification degree $e_i$ and let $p^{\text{max}}$ be the largest power of $p$ dividing any of the ramification degrees. Let $D_p$ be the union of all ramification divisors $D_i$ whose ramification degrees are divisible by $p^{\text{max}}$, then $p_a(D_p)\geq 1$, where $p_a(D_p)$ denotes the arithmetic genus of $D_p$. 
\end{proposition}

Let $D$ be a ramification divisor of some order over the rational ruled surface $\Fb_n$. If $p$ is a prime number dividing any ramification degree and $D_p=a_pC_0+b_pF$ for some $a_p$ and $b_p$, then by genus formula and Proposition \ref{genus for p-max} we have
\begin{equation}\label{genus}
2p_a(D_p)=(a_p-1)(2b_p-na_p-2)\geq 2.
\end{equation}

\begin{remark}\label{ramification degree a=3}
Let $\Xc$ denote a minimal terminal order over a rational ruled surface $Z$ and let $D=\cup D_i\equiv 3C_0+bF$ be its ramification divisor. Then the ramification degrees of the divisors intersecting $F$ is $2$. Further, by Proposition \ref{genus for p-max} and Equation \ref{genus} all other ramification degrees divide $2$, therefore they are all equal.
\end{remark}

\subsubsection*{Z=$\Pb^1\times\Pb^1$}\label{subsection Z=P*P}
Let $\Xc$ be a minimal terminal order over $Z=\Pb^1\times\Pb^1$ and let $D=\cup D_i\equiv aC_0+bF$ be the ramification divisors and the ramification degrees. Then   $2\leq a,b\leq 3$ and all the ramification degrees $e_i=e$. Further, $e=2$ if $2C_0+2F<[D]$. Considering Proposition \ref{centre of del pezzo on PP}, we see that if $\Xc$ is a minimal TAdPO over $\PP$ then it is actually del Pezzo  and the classification is as before.

\subsubsection*{Z=$\Fb_1$}\label{subsection Z=F1}
Now assume that $\Xc$ denotes a minimal TAdPO over $Z=\Fb_1$. And let $\cup D_i$ $\{e_i\}$ denote its ramification divisors and the ramification degrees.  Also let $C_0$ be the minimal section with  $C_0^2=-1$, and $F$ a fixed fibre. For every $i$ and for suitable $a_i$ and $b_i$ we have 
$D_i\equiv a_iC_0+b_iF$. We also  set $D:=\cup D_i\equiv aC_0+bF$ for non-negative integers $a$ and $b$. 

 Since $\Xc$ is a minimal terminal order and $C_0^2<0$, then $K_{\Xc}C_0\geq 0$. On the other hand, $\Xc$ is almost del Pezzo, so $K_{\Xc}C_0\leq 0$. Thus $K_{\Xc}C_0=0$.
 \begin{align*}
  0=K_{\Xc}C_0 & =\left(-2C_0-3F+\sum(1-\frac{1}{e_i})(a_iC_0+b_iF)\right)C_0\\
  & =2-3+\sum(1-\frac{1}{e_i})(-a_i+b_i)
 \end{align*}
  \begin{equation}\label{F_1 a_i,b_i}
   \Rightarrow~~~  \sum(1-\frac{1}{e_i})b_i=\sum(1-\frac{1}{e_i})a_i+1
  \end{equation}

\begin{proposition}
The minimal section $C_0$ on a minimal TAdPO over $\Fb_1$ is unramified.
\end{proposition}
\begin{proof}
Suppose instead that the ramification divisors are $D_1, \ldots, D_n$ with $D_1=C_0$, so $a_1=1$ and $b_1=0$. Equation  \ref{F_1 a_i,b_i} states
\begin{equation}\label{F_1 a_i,b_i 2}
\sum_1^n(1-\frac{1}{e_i})(b_i-a_i)=1.
\end{equation}
Let $\widetilde{D}\rightarrow D_1\simeq\Pb^1$ be the degree $e_1$ cover corresponding to the ramification of the order on $D_1$. Now all the $e_i|e_1$ so the secondary ramification indices on $D_1$ are $e_2, \ldots, e_n$ which occur $D_i.D_1=b_i-a_i$ times. Hence Riemann-Hurwitz and Equation \ref{F_1 a_i,b_i 2} give
\begin{align*}
\deg K_{\widetilde{D}}&=-2e_1+e_1\sum_{i=2}^n\left(1-\frac{1}{e}\right)(b_i-a_i)\\
&=-e_1-1-e_1\left(1-\frac{1}{e_1}\right)+e_1\sum_{i=2}^n\left(1-\frac{1}{e}\right)(b_i-a_i)\\
&=-e_1-1-e_1\sum_{i=1}^n\left(1-\frac{1}{e}\right)(b_i-a_i)=-1.
\end{align*}
This contracts the fact that $\deg K_{\widetilde{D}}$ is even.
\end{proof}
 
 \begin{theorem}\label{almost del pezzo on F1}
 Let $\Xc$ be a minimal TAdPO over $\Fb_1$ and let $D=\cup D_i\equiv aC_0+bF$ be its ramification divisors.  Then 
 \begin{enumerate}
 \item
 $D\equiv 2C_0+4F$~or
 \item
 $D\equiv 3C_0+5F$.
 \end{enumerate}
 Furthermore, the  ramification degrees are all equal to $2$.
 \end{theorem}
 
 \begin{proof} 
In Equation \ref{F_1 a_i,b_i 2}, the intersection numbers $b_i-a_i\geq0$ since $C_0$ is unramified. Hence the non-zero terms in \ref{F_1 a_i,b_i 2} can only come from summing $\frac{1}{2}$'s corresponding to $e_i=2$. In particular, all ramified fibres have ramification index $2$. Suppose the ramified (bi)-section has ramification index greater than $2$. Then we can find some prime $p$ such that $D_p$ is just the ramified bi-section. We may assume $a=2$ so Equation \ref{genus} gives $a_p=2$ and $b_p\geq3$. Hence, at least one of the ramified bi-sections has $b_i-a_i>0$ and, giving rise to a non-zero term in Equation \ref{F_1 a_i,b_i 2}, must have ramification index $2$, a contradiction.

Now that all the indices are equal to $2$, Equation \ref{F_1 a_i,b_i} becomes
\[
\frac{1}{2}b=\frac{1}{2}a+1.
\]
Therefore, if $a=2$ or $a=3$, then respectively $b=4$ and $b=5$
 \end{proof}

\subsubsection*{Z=$\Fb_2$}\label{subsection Z=F2}
Let $C_0$ be the minimal section with $C_0^2=-2$ and $F$ a fixed fibre of the ruled surface $\Fb_2$.  Recall that the canonical divisor is $K_Z\equiv -2C_0-4F$. For suitable $a_i$ and $b_i$ we have 
$D_i\equiv (a_iC_0+b_iF)$. We also  set $D:=\cup D_i$ and $aC_0+bF=\cup (a_iC_0+b_iF)$ for non negative integers $a$ and $b$.

\begin{proposition}\label{F_2 3 cases}
Consider a minimal TAdPO over $\Fb_2$ ramified on a bisection. Then one of the following holds:
\begin{enumerate}
\item
$C_0$ is ramified and the secondary ramification on $C_0$ occurs at precisely two points,
\item
$D$ consists of two sections, both linearly equivalent to $C_0+2F$,
\item
$D\sim 2C_0+4F$ consists of a single bi-section.
\end{enumerate}
\end{proposition}
\begin{proof}
Let $D\equiv aC_0+bF=\cup (a_iC_0+b_iF),~\{e_i\}$ be a ramification configuration for a minimal AdPO over $\Fb_2$. 
Similar to the case $\Fb_1$, we know $K_{\Xc}C_0=0$. Therefore,
 \begin{align}\label{F2 divisors}
  0=K_{\Xc}C_0 & =\left(-2C_0-4F+\sum_{i=1}^n(1-\frac{1}{e_i})(a_iC_0+b_iF)\right)C_0\nonumber\\
  & =4-4+\sum_{i=1}^n(1-\frac{1}{e_i})(-2a_i+b_i)\nonumber\\
&=\sum_{i=1}^n(1-\frac{1}{e_i})(b_i-2a_i).
 \end{align}
If $C_0$ is unramified, then all the intersection numbers $b_i-2a_i\geq0$, so the terms in Equation \ref{F2 divisors} are also non-negative and hence zero. This gives cases 2 and 3. We now assume that $D_1=C_0$ is ramified. Again let $\widetilde{D}\rightarrow D_1$ be the cyclic cover describing the ramification of the order at $D_1=C_0$. Now Equation \ref{F2 divisors} and Riemann-Hurwitz give
\[
2e_1\left(1-\frac{1}{e_1}\right)=e_1\sum_{i=1}^n\left(1-\frac{1}{e_i}\right)D_i.D_1=\deg K_{\widetilde{D}}+2e_1.
\]
Hence $\deg K_{\widetilde{D}}=-2$ and $\widetilde{D}\simeq\Pb^1$. But any cyclic cover $\Pb^1\rightarrow\Pb^1$ is ramified at precisely two points so we are in case 1.
\end{proof}
 \begin{theorem}\label{almost del pezzo on F2}
 Let $\Xc$ be a minimal TAdPO over  $\Fb_2$ and let $D,~\{e_i\}$ be its ramification divisors and degrees. Then
 \begin{enumerate}
 \item
 $D\equiv 2C_0+4F$~or
 \item
 $D\equiv 3C_0+6F$.
 \end{enumerate}
 Furthermore, the ramification degrees are all equal and in the second case they are $2$.
 \end{theorem}
 
 \begin{proof}
 Let $D\equiv aC_0+bF=\cup (a_iC_0+b_iF),~\{e_i\}$ be a ramification configuration for a minimal AdPO over $\Fb_2$. 
If $a=3$,  by Remark \ref{ramification degree a=3} we know all the ramification degrees are $2$. So, by Equation \ref{F2 divisors} we get $b=6$.
If $a=2$, all ramified sections have the same ramification degrees and the ramification degrees of the ramified fibres are given by secondary ramification degrees, which by case 1 in Proposition \ref{F_2 3 cases} are equal to the ramification degree of $C_0$. So all the ramifications degrees are the same, and again, by Equation \ref{F2 divisors} we get $b=4$.

\end{proof}
Considering  minimal TAdPO classified in this section, one can observe the following corollary.
\begin{corollary}
Let $\Xc$ be a minimal TAdPO over $Z$. Then all the degrees are equal, say to $e$, and there is an effective divisor $M$ such that the ramification divisor $D -K_Z+M$. In particular, $K_{\Xc}=\frac{1}{e}(K_Z+(e-1)M)$.
\end{corollary}

\section{Classification of CdPOs}\label{Classification of CdPOs}
 In this section we will give the classification of all CdPOs. Let $\Xc$ be a canonical order. In the previous section, we showed that there exists  the following diagram
 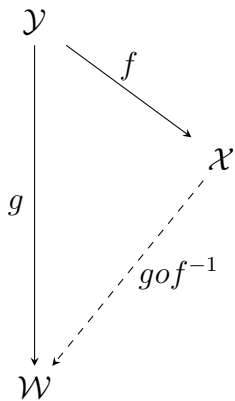
\begin{figure}[H]
 \begin{center}
 \begin{tikzpicture}
  \matrix (m) [matrix of math nodes,row sep=3em,column sep=4em,minimum width=2em]
  {
     \Yc &  \\
      & \Xc \\
      & \\
      \Wc& \\};
  \path[-stealth]
    (m-1-1) edge  node [above] {$f$} (m-2-2)
            edge node [left] {$g$} (m-4-1)
   (m-2-2) edge [dashed] node [right] {$gof^{-1}$} (m-4-1);
    \end{tikzpicture}
 \caption{Resolution of canonical orders to the minimal model}
 \label{resolution diagram for orders}
\end{center}
\end{figure}
\noindent where $f:\Yc\rightarrow \Xc$ is the unique minimal resolution for $\Xc$; meaning that $\Xc$ is resolved to a terminal order $\Yc$ over a smooth surface $Z_{\Yc}$, and $\Wc$ is a minimal terminal order over $Z_{\Wc}=\Pb^2, \Pb^1\times\Pb^1, \Fb_1$, or $\Fb_2$. Further, if $\Xc$ is a del Pezzo order, then $\Wc$ is almost del Pezzo, however $W$ is actually del Pezzo when $Z_{\Wc}=\Pb^2$ .  So we can blow up a minimal terminal (almost) del Pezzo order $\Wc$ over $Z$ to classify CdPOs. Once the order $\Wc$ is blown up by $f:\Yc\rightarrow \Wc$, then $\Yc$ should be an almost del Pezzo order. Moreover, we need to have a $K_{\Yc}$-zero curve in order to do the contraction  $f:\Yc\rightarrow \Xc$ where $\Xc$ is a canonical del Pezzo order.

For the special case where the ramification is anti-canonical we have the following result in order to find $k_{\Yc}$-zero curves.
\begin{proposition}\label{Z iff X a}
Let $\Wc$ be a minimal terminal order over $Z_{\Wc}$ with anti-canonical ramification divisors. Then  $\Wc$ is almost del Pezzo if and only if $Z_{\Wc}$ is. Furthermore, with above setting in Figure  \ref{resolution diagram for orders},   if $\Yc$ is obtained from blowups of $\Wc$ at points of the ramification divisor, $K_{\Yc}$-zero curves are preciesely the $(-2)$-curves on $Z_{\Yc}$.
\end{proposition}
\begin{proof}
The first part of the proposition can be seen like Proposition \ref{Z iff X}. Now let $f:\Yc\rightarrow \Xc$ be the minimal resolution of $\Xc$ and $g:\Yc\rightarrow\Wc$ represent blowups some points of the anti-canonical ramification divisor of $\Wc$. Also, let $C$ be an $f$-exceptional curve. Then by  \cite[Proposition 6.2]{chan2008canonical} the genus of $C$, $g_C$ is zero. Therefore, by genus formula, we have $-2=K_{Z_{\Yc}}C+C^2=eK_{\Yc}C+C^2$, which proves the last part of the result. Note that $K_{Z_{\Yc}}=eK_{\Yc}$ as the blowups are at points of the ramification locus.
\end{proof}
In Proposition \ref{Y dP M}  criteria are given to check if an order is del Pezzo or not. We can generalize the result for AdPOs as follows. It also gives a good description of $K_{\Yc}$-zero curves, which will be used in this Section to classify CdPOs.
\begin{proposition}\label{Y adP M}
If $\Yc$ is almost del Pezzo, then the following condition holds on the blowup $f:Z'\rightarrow Z$ at points $\Sigma\subset D$. Let $C\subset Z'$ be an effective curve. Then the multiplicity $m=mult_{\Sigma}f_*C\leq2-(e-1)M.f_*C+f_*C^2$. Also, $\Sigma$ contains no point on a $(-2)$-curve. Further,  $K_{\Yc}$-zero curves are precisely the curves $C$ such that $f_*C$  have exactly $2-(e-1)M.f_*C+f_*C^2$ blown up points.
\end{proposition}
\begin{proof}
We only need to replace strict inequalities in Proposition \ref{Y dP M} by inequalities.
\end{proof}

\subsection{CdPOs Obtained from TdPOs over $\Pb^2$}\label{Canonical Orders on P2}

Let $D$ be the ramification divisors over $\Pb^2$ corresponding to a TdPO $\Wc$. Then $3\leq \deg(D)\leq 5$. We firstly let $D$ be of degree $3$. So all the ramification degrees are the same, say $e$. 

Let $\Wc$ be a TdPO over $\Pb^2$ and let $f:\Yc\rightarrow \Wc$ be a blowup at a point $p$ not in $D$, then  $\Yc$ is almost del Pezzo if $e=2$; however it actually remains del Pezzo, so there is no $K_{\Yc}$-zero curve. So Let  $f:\Yc\rightarrow \Wc$ denote a blowup at a point $p$ out of $D$  with the exceptional $E_p$ and then blowing up a point $q\in E_p$ with the exceptional $E_q$. Then $E_p^2=-2$ and $(E_q)^2=-1$. We have $K_{\Yc}^2=K_{\Wc}^2-2=\frac{1}{4}$ and further,
\[
K_{\Yc}\equiv f^*(K_{\Wc})+2E_q+E_q.
\]
Let $C=C_0+aE_p+bE_q$ be an effective curve in $Z_{\Yc}$ where neither $C_0-E_p$ nor $C_0-E_q$ is effective. Then
\begin{align*}
K_{\Yc}D&=(f^*(K_{\Wc})+2E_{q}+E_p)C\\
&=f^*(K_{\Wc})C+(2E_{q}+E_p)C\\
&=K_{\Wc}f_*C+(2E_{q}+E_p)C\\
&=\dfrac{-3}{2}d+(2E_{q}+E_p)(C_0+aE_p+bE_{q})\\
&=\dfrac{-3}{2}d+(2E_{q}+E_p)C_0+(2E_{q}+E_p)(aE_p+bE_{q})\\
&\leq\dfrac{-3}{2}d+d-b\\
&\leq 0.
\end{align*}

Note that the only case that $(K_{\Yc}+\Delta_{\Yc})C=0$ is when $C=aE_p$ for any positive integer $a$.

If we blow up one more point, then $K^2\leq 0$, so the order would not be (almost) del Pezzo.  

Now let $f:\Yc\rightarrow \Wc$ refer to blowing up two points  $p\notin D$ and $q\in D$. Then
\begin{align*}
K_{\Yc}&\equiv f^*(K_{\Wc})+E_p+\frac{1}{2}E_q;\\
K_{\Yc}^2&=\frac{9}{4}-1-\frac{1}{4}>0;\\
K_{\Yc}D&=(f^*(K_{\Wc})+E_p+\frac{1}{2}E_q)D\\
&=-\frac{3d}{2}+E_pC+\frac{1}{2}E_qC-a-\frac{b}{2}\\
&\leq -a-\frac{b}{2},
\end{align*}
where $D\equiv C+aE_p+bE_q$ and $d=\deg(f_*C)$. In order for the equality to hold, $a$ and $b$ should be $zero$ and multiplicity of $C$ at both $p$ and $q$ should be $d$. Therefore, $D:=\tilde{\ell}$ where $\ell$ is the line going through $p$ and $q$.

\begin{definition}\label{almost general position P2}
Let $f:Z\rightarrow \Pb^2$ be a sequence of blowups at points $\Sigma=\{p_1, \cdots, p_n\}$ in order and let  $E_1, \cdots, E_n$  be the corresponding exceptional curves, $1\leq i\leq n$. Note that the points are not in the same surfaces; however, they are all in blowups of $\Pb^2$ and we allow infinitely near points.  The set of points $p_1, \cdots, p_n$ is in almost general position if:

 \begin{enumerate}
  \item No four points (counting the multiplicities) are on a line.
  \item No seven points (counting the multiplicities) are on a conic.
  \item No point of a $(-2)$-exceptional curve is blown up.
 \end{enumerate}
\end{definition}

\begin{theorem}\label{canonical del Pezzo orders by P2}
 Let $\Wc$ be a minimal TAdPO over $\Pb^2$ with ramification divisor $D$ of degree $3$. Also let $g:\Yc\rightarrow \Wc$ represent the blowups of $\Wc$. Then, each of the following cases gives a contractible  $K_{\Yc}$-zero curve $E$ such that if $f:\Yc\rightarrow\Xc$ contracts $E$, then $\Xc$ is a CdPO. These actually classify all $k_{\Yc}$-zero curves.
 \begin{enumerate}
\item
  Blowing up a point $p\in D$ twice, to get the exceptional curves $E_1$ and $E_2$, where $E_1^2=-2$ and $E_2^2=-1$; $E:=E_1$.
  \item
  Blowing up  $3$ points in $D$, counting multiplicities, where all the points belong to  a line $l$; $E:=l$.
 \item
  Blowing up  $6$ points in $D$, counting multiplicities, where all the points belong to  a conic $C$; $E:=C$.
  \item 
  Blowing up  points $p\in D$ and $q\notin D$ where $e=2$; $E$ is the line going through $p$ and $q$.
  \end{enumerate}
\end{theorem}
\begin{proof}
The first three cases are clear by Proposition \ref{Z iff X a}. For the last case see Proposition \ref{blowup 2 points out of D}.
\end{proof}

\begin{theorem}
 Let $\Wc$ be a minimal TAdPO over $\Pb^2$. Let $D$ be the  ramification divisor of degree $4$. Let $g:\Yc\rightarrow \Wc$ represent blowups of $\Wc$ at points $\Sigma=\{p_1, \cdots, p_n\}$. Then $\Yc$ is almost del Pezzo if and only if $\Sigma\subset D$, $n\leq 3$ and no three points of $\Sigma$ are collinear.
\end{theorem}
\begin{proof}
By Theorem \ref{blowups degree 4} we know all the blowups should be at points in $D$ and also $e=2$. Let  $D\sim-K_{\Pb^2}+H$ be the ramification divisors in $\Pb^2$ for a line $H$. Equations \ref{self intersect P^2} shows $(K_{\Pb^2}+\Delta_{\Pb^2})^2=1$.
 So there are only $3$ blowups allowed. Now let $f:Z_{\Yc}\rightarrow \Pb^2$ be blowups at $n$ points, $n\leq3$, with exceptional curves $E_i$, $1\leq i\leq n$. By Remark \ref{effective cone} we see that the effective cone is generated by exceptional curves and $\widetilde{H}$ for some line $H\in \Pb^2$. Since no three point are collinear
\[
mult_{\Sigma}H\leq2=2-H^2+H^2
\]
and for any point $p\in\Sigma$,
\[
mult_{\Sigma}p\leq2-H.p+p^2.
\]
\end{proof}

The last case to check for the orders on the projective plane is the orders with ramification divisors of degree $5$. But looking at calculations in the proof of Theorem \ref{blowups degree 5}, we see that the first blowup $g:\Yc\rightarrow \Wc$ results in $K_{\Yc}^2\leq 0$. So the order $\Yc$ can not be almost del Pezzo.

\subsection{CdPOs Obtained from TdPOs over Rational Ruled Surfaces }
\subsubsection*{The case $Z=\PP$}
Let $\Wc$ be a minimal TAdPO over $\PP$ but we see that in this case $\Wc$ is actually del Pezzo. Let $D$ denote the ramification divisors corresponding to $\Wc$ and let $C_0$ and $F$ denote two perpendicular fixed fibres. Then $2C_0+2F\leq [D]\leq3C_0+3F$ and $\Delta=\left(1-\dfrac{1}{e}\right)D$. We know that $e=2$ if $2C_0+2F< [D]$.  Recall that the canonical divisor of $\Wc$ is as the following.
\begin{align*}
K_{\Wc}=K_{\PP}+\Delta_{\PP}&\equiv (-2C_0-2F)+\left(1-\dfrac{1}{e}\right)(aC_0+bF),\\
&=\left(a-2-\dfrac{a}{e}\right)C_0+\left(b-2-\dfrac{b}{e}\right)F,
\end{align*}
for suitable $a$ and $b$.
Recalling the diagram in Figure \ref{resolution diagram for orders}, we seek to find a $K_{\Yc}$-zero curve $E\in Z_{\Yc}$. This lets us blow down $\Yc$ to a CdPO $\Xc$ by contracting $E$. Before that, we need the following definition.

\begin{definition}\label{almost general position P1}
Let $\Sigma$ be a set of points in $\PP$. The points of $\Sigma$ are in almost general position if $|\Sigma|<8$ and any irreducible curve of the form $aC_0+bF$ contains no more than  $2(a+b)$ of the points. If $C$ is an irreducible $(a,b)$-curve, we say $C$ is in $\Sigma$-almost general position if it contains exactly $2(a+b)$ points of $\Sigma$
\end{definition}

\begin{theorem}\label{canonical del Pezzo orders by P11}
Let $\Wc$ be a minimal TAdPO over $\PP$ with ramification divisors $D$ and ramification degree $e$. Let $g:\Yc\rightarrow\Wc$ be a sequence of blowups at points $\Sigma=\{p_1, \cdots, p_n\}$. Then each of the followings gives a contractible $K_{\Yc}$-zero curve $E$ such that if $f:\Yc\rightarrow\Xc$ contracts $E$, then $\Xc$ is a CdPO. And this actually classifies all CdPOs.
\begin{enumerate}
\item
$D\equiv2C_0+2F$, $e=2$, and $\Sigma=\{p\}$ is a single point, where $p\notin D$. Then $E$ is the proper transform of any fibre (in any direction) passing through $p$.
\item
$D\equiv3C_0+2F$, $e=2$, and $\Sigma=\{p\}$ is a single point, where $p\in D$. Then $E$ is the proper transform of  any fibre in $[F]$ passing through $p$.
\item
$D\equiv3C_0+3F$, $e=2$, and $\Sigma=\{p\}$ is a single point, where $p\in D$. Then $E$ is the proper transform  of any fibre (in any direction) passing through $p$.
\item
$D\equiv2C_0+2F$, $e$ is free, and $\Sigma\subset D$ is a set of points in almost general position. Then $E$ is the proper transform of the blowup of  any irreducible curve in  $\Sigma$-almost general position.
\end{enumerate}
Since in $\PP$ irreducible $(a,b)$-curves  and $(b,a)$-curves are isomorphic, we only worked with one case.

\end{theorem}
\begin{proof}
For case \textit{1} see Proposition \ref{blowup p notin PP}. Cases \textit{2} and \textit{3} follow from Proposition \ref{Y adP M}. Case \textit{4} we only need to show that any curve in  $\Sigma$-almost general position is a $(-2)$-curve then we use Proposition \ref{Z iff X a}. Let $C$ be an irreducible $(a,b)$-curve containing $2(a+b)$ points of $\Sigma$. Then $a+b<4$ and $C$ only can have bi-degrees $(1,0)$, $(1,1)$, and $(2,1)$. It is easy to check that blowing up these curves at $2$, $4$, and $6$ points respectively give $(-2)$-curves.
\end{proof}

\subsubsection*{The cases $Z=\Fb_1$}
Let $\Wc$ be a TAdPO over the ruled  surface $\Fb_1$ with  ramification divisors  $D$. By Theorem \ref{almost del pezzo on F1},  $D\equiv 2C_0+4F$ or $D\equiv 3C_0+5F$ and the ramification degrees are all equall to $2$. Further we have the following equations for the canonical divisors.
\begin{align*}
 D\equiv 2C_0+4F:~~&K_{\Wc}=-2C_0-3F+\dfrac{1}{2}\left(2C_0+4F\right)=-(C_0+F)\\
 D\equiv 3C_0+5F:~~&K_{\Wc}=-2C_0-3F+\dfrac{1}{2}\left(3C_0+5F\right)=-\dfrac{1}{2}(C_0+F)		
\end{align*}

\begin{lemma}\label{non del pezzo f1 2}
Let $\Wc$ be a minimal TAdPO over $\Fb_1$ with ramification divisors $D\equiv 3C_0+5F$. Let $g:\Yc\rightarrow \Wc$ be any blowup of $\Wc$. Then Y is not almost del Pezzo.
\end{lemma}
\begin{proof}
Let $g:\Yc\rightarrow\Wc$ denote a blowup at a point $p$. Then we have the following equations
 \begin{align*}
  p\notin D:~~K_{\Yc}^2&=\left(g^*(K_\Wc)+E\right)^2\\
  &=\dfrac{1}{4}(C_0+F)^2-1<0\\
  p\in D:~~K_{\Yc}^2&=(g^*(K_\Wc)+\dfrac{1}{2}E)^2\\
  &=\dfrac{1}{4}(C_0+F)^2-\dfrac{1}{4}=0
 \end{align*}
 which can not occur for almost del Pezzo surfaces.
\end{proof}

Reviewing calculations in the proof of Theorem \ref{almost del pezzo on F1} we see that if $\Wc$ is a TAdPO over Hirzebruch surface $\Fb_1$ with  ramification divisors $D$, then $C_0$ is a $K_{\Wc}$-zero curve. Therefore contracting $C_0$ gives a blowdown to a CdPO. 

If $D\equiv 3C_0+5F$, then by Lemma \ref{non del pezzo f1 2} there is no blowup to an almost del Pezzo order, meaning that $C_0$ is the only contractible curve. So $\Xc$ is a canonical del Pezzo order over $\Pb^2$ with ramification divisors of degree $5$ and ramification degree $e=2$.

If $D\equiv 2C_0+4F$, it needs a more detailed discussion. We claim that if $\Wc$ is blown up to a TAdPO, then the blowups are at  points of $D$. Otherwise if we blowup a point $p\notin D$ then 
 \begin{align*}
 K_{\Yc}^2&=\left(g^*(K_\Wc)+E\right)^2\\
  &=(-C_0-F)^2-1\\
  &=0,
 \end{align*}
 which is not true for almost del Pezzo orders.

\begin{definition}\label{almost general position F1}
Let $\Sigma=\{p_1, \cdots, p_n\}$ be a set of points in $\Fb_1$. $\Sigma$ is in almost general position if $|\Sigma|<8$ and any irreducible curve of the form $aC_0+bF$ contains no more than $2+a(2b-a-1)$ points of $\Sigma$.
\end{definition}

Now we classify del Pezzo orders with canonical singularities for which the minimal terminal del Pezzo order is over $\Fb_1$ and $D\equiv 2C_0+4F$. In the next two theorems, we  assume that the blowups are done in orders. Namely, the $i$-th blowup is at the point $p_i$ for $1\leq i\leq n$ and $\Sigma\subset D$. Note that if $p_1$ and $p_2$ are infinitely near, then depending on $p_1$ if it is a singular point or not $E_1$ may be a ramification divisor. Since $p_2$ must be in $D$ if $E_1\in D$, then any point on $E_1$ can be blown up, but if $E_1$ is not a ramification divisor, then $p_2$ is the only point of the intersection of $E_1$ and $D$.

\begin{theorem}
Let $\Wc$ be a minimal TAdPO over $\Fb_1$ with ramification divisors $D\equiv 2C_0+4F$. Also Let $g:\Yc\rightarrow\Wc$ be a sequence of blowups at points $\Sigma=\{p_1, \cdots, p_n\}\subset D$. Then $\Yc$ is almost del Pezzo if and only if $n\leq 3$ and $\Sigma$ is in almost general position.
\end{theorem}
\begin{proof}
Let $g:\Yc\rightarrow\Wc$ be a sequence of blowups at points $\Sigma=\{p_1, \cdots, p_n\}$. Since all the blowups are at points of $D$ and $e=2$, each blowup reduces the self-intersection $K_{\Wc}$ by $\dfrac{1}{4}$. Moreover, $K_{\Wc}^2=(C_0+F)^2=1$. Then  $K_{\Yc}^2>0$ implies $n<4$.
For the rest of the proof we first show that if $\Sigma$ is in almost general position, then $\Yc$ is almost del Pezzo. And then we show that if $\Sigma$ is not in almost general position, then $\Yc$ is not almost del Pezzo.
Let $C=\widetilde{aC_0+bF}+r_1E_1+r_2E_2+r_3E_3$ be an effective curve in $Z_{\Yc}$, where $r_2$ and $r_3$ can be zero depending on the number of blowups $n$. Then
\[
K_{\Yc}=g^*(-C_0-F)+a_1E_1+a_2E_2+a_3E_3,
\]
where $a_i\in\left\{\dfrac{1}{2},1,\dfrac{3}{2}\right\}$ depend on the blowups and the tree of exceptional curves.
\begin{align*}
K_{\Yc}.C&=(g^*(-C_0-F)+a_1E_1+a_2E_2+a_3E_3).C\\
&=(-C_0-F)(aC_0+bF)+(a_1E_1+a_2E_2+a_3E_3)(aC_0+bF)\\
&+\underbrace{\sum_1^3a_iE_i.\sum_1^3r_iE_i}_{\leq0}\\
&\leq-b+b\underbrace{(a_1E_1+a_2E_2+a_3E_3).F}_{\leq1}\\
&\leq 0.
\end{align*}
Now if $\Sigma$ is not in almost general position, then at least one of the three conditions fails. If $g:\Yc\rightarrow\Wc$ denotes a blowup at a point $p\in C_0$. Then
 \begin{align*}
 K_{\Yc}\tilde{C_0}&=\left(g^*(K_\Wc)+\dfrac{1}{2}E\right)\tilde{C_0}\\
 &=g^*(K_\Wc).\tilde{C_0}+\dfrac{1}{2}E.\tilde{C_0}\\
 &=(-C_0-F).C_0+\dfrac{1}{2}\\
 &=\dfrac{1}{2}.
 \end{align*}
Which is against the definition of almost del Pezzo surfaces.

If there is a fibre $F$ with multiplicity more than $2$ at $\Sigma$, then
 \begin{align*}
 K_{\Yc}\tilde{F}&=\left(g^*(K_\Wc)+\sum_ia_iE_i\right)\tilde{F}\\
 &=g^*(K_\Wc).\tilde{F}+\left(\sum_ia_iE_i\right)\tilde{F}\\
 &\geq(-C_0-F).F+\dfrac{3}{2}\\
 &=\dfrac{1}{2}.
 \end{align*}
 And finally, there is no blowup at a point of a $(-2)$-exceptional curve for the obvious reason that there is no $(-3)$-curve in the resolution of canonical orders.
 
 \end{proof}
\begin{theorem}\label{canonical del Pezzo orders by F1}
Let $\Wc$ be a minimal TAdPO over $\Fb_1$ with ramification divisors $D\equiv 2C_0+4F$. Also Let $g:\Yc\rightarrow\Wc$ be a sequence of blowups at points $\Sigma=\{p_1, \cdots, p_n\}$ in almost general position. Each of the following is a $K_{\Yc}$-zero curve $E$ such that if $f:\Yc\rightarrow\Xc$ contracts $E$, then $\Xc$ is a CdPO.
\begin{enumerate}
\item
The section $C_0$.
\item
Any fibre $\tilde{F}$, where multiplicity of $F$ at $\Sigma$ is $2$.
\item
An exceptional curve $E$, where $E^2=-2$.
\end{enumerate}
\end{theorem}
\begin{proof}
Cases \textit{1} and \textit{2} follow from Proposition \ref{Y adP M} with $M=F$.
For case \textit{3}  let $E_1$ and $E_2$ be a tree of exceptional curves where  $E_1^2=-2$. Then
\begin{align*}
K_{\Yc}.E_1&=\left(g^*(-C_0-F)+\dfrac{1}{2}E_1+E_2\right).E_1\\
 &=0+\dfrac{1}{2}(-2)+1=0
\end{align*}
\end{proof}
\subsubsection*{The cases $Z=\Fb_2$}
Let $\Wc$ be a a TAdPO over the surface $\Fb_2$ with  ramification divisors  $D$. By Theorem \ref{almost del pezzo on F2},  $D\equiv 2C_0+4F$ or $D\equiv 3C_0+6F$ and the ramification degrees are all equal to $e$ and further in the second case  $e=2$. Then we have the following equations for the canonical divisors.
\begin{align}\label{2*}
 D\equiv 2C_0+4F:~~&\left\{\begin{array}{ll}K_{\Wc}&=-\dfrac{1}{e}\left(2C_0+4F\right)\\
 K_{\Wc}^2&=\dfrac{8}{e^2}\\\end{array}\right.
\end{align}

\begin{align}\label{3*}
 D\equiv 3C_0+6F:~~&\left\{\begin{array}{ll}K_{\Wc}&=-\dfrac{1}{2}C_0-F\\	
 K_{\Wc}^2&=\dfrac{1}{2}\\\end{array}\right.
\end{align}
We want to know how many and what types of blowups give a $K_{\Yc}$-zero curve where we denote the sequence of blowups by $f:\Yc\rightarrow\Wc$. We start to classify the case $D\equiv 3C_0+6F$ as it is very restrictive. By  Equation \ref{3*}  we know only one single blowup keeps the order almost del Pezzo and it has to be at a point $p\in D$. By calculations in the proof of Theorem \ref{almost del pezzo on F2}  we know that $K_{\Yc}.C_0=0$ and so $p\notin C_0$. Let $C=\widetilde{bF}+rE$ be an effective curve in $Z_{\Yc}$. Then
\begin{align*}
K_{\Yc}C&=\left(f^*\left(-\dfrac{1}{2}C_0-F\right)+\dfrac{1}{2}E\right)C\\
&=f^*\left(-\dfrac{1}{2}C_0-F\right).C+\dfrac{1}{2}E.C\\
&=\left(-\dfrac{1}{2}C_0-F\right).bF+\dfrac{1}{2}E.(\widetilde{bF}+rE)\\
&=-\dfrac{b}{2}+\dfrac{b}{2}-\dfrac{r}{2}.
\end{align*}
So if $g:\Yc\rightarrow\Xc$ contracts $C_0$ or the fibre $F$ where the blowup is at a point $p\in F$, then $\Xc$ is a CdPO.

The classification of minimal TAdPOs over $\Fb_2$ with ramification divisors $D\equiv 2C_0+4F$ is more enormous. This is actually in two extents, a blowup at a point out of $D$ and on the other hand, more blowups at points in $D$ keep the order almost del Pezzo.

\begin{definition}\label{Z2 Y adp}
Let $\Sigma=\{p_i\}_i$ be a set of points in $\Fb_2$. $\Sigma$ is in almost general position if $|\Sigma|<8$ and no more than $2+a(2b-2a)$ lie on an irreducible curve of the form $aC_0+bF$.
\end{definition}

\begin{proposition}\label{almost general theorem F2}
Let $\Wc$ be a minimal TAdPO over $\Fb_2$ with ramification divisor $D\equiv 2C_0+4F$. Also let $f:\Yc\rightarrow \Wc$ represent a sequence of blowups at the points $\Sigma=\{p_i\}_i$. Then $\Yc$ is almost del Pezzo if and only if the points are in almost general position and one of the followings occurs
\begin{enumerate}
\item
If there is a blowup at a point $p_i\notin D$, then $n\leq 4$ and all other points are in $D$, and  non of them should lie on the same fibre as the one $p_i$ does.
\item
$n\leq7$ and  $\Sigma\subset D$.
\end{enumerate}
\end{proposition}
\begin{proof}
Let $\Yc$ be almost del Pezzo. Use the fact that the centre of $\Yc$ is almost del Pezzo to make sure that $\Sigma$ satisfies the conditions of Definition \ref{Z2 Y adp}. By Equation \ref{2*} we know that if there is any blowup at a point out of $D$, then $e=2$ and therefore $K_{\Wc}^2=2$. Further each blowup at a point out of $D$ reduces the self-intersection of the order  by one, thus only one  point $p_i$ can lie out of $D$. Then after blowing up a point out of $D$, there can be at most  three more blowups and they are all at points of $D$ as each such blowup reduces the self-intersection of the order  by $\dfrac{1}{4}$.

Let $g:\Yc\rightarrow \Wc$ represent a sequence of blowups at  points $p_1\notin D$ and $\{p_2, p_3, p_4\}\subset D$ with the exceptional curves $\{E_i\}$ respecting indices. If by contradiction $p_1$ and $p_2\in F$ for some fibre $F$, then
\begin{align*}
K_{\Yc}.\tilde{F}&=\left(f^*\left(-C_0-2F\right)+E_1+\dfrac{1}{2}\sum_{i=1}^3a_iE_i\right)\tilde{F}\\
&=-1+E_1\tilde{F}+\dfrac{1}{2}\underbrace{\sum_{i=1}^3a_iE_i\tilde{F}}_{>0}>0,
\end{align*}
which is a contradiction. 

Now let $\Sigma$ be in almost general position and case $1$ occurs. Then obviously $K_{\Yc}^2>0$. Further let $C$ be a curve in $Z_{\Yc}$ if $C=E_j$ for some $j$, then
\begin{align*}
K_{\Yc}.E_j&=\left(f^*\left(-C_0-2F\right)+E_1+\dfrac{1}{2}\sum_{i=1}^3a_iE_i\right)E_j\\
&=\left(E_1+\dfrac{1}{2}\sum_{i=1}^3a_iE_i\right)E_j<0.
\end{align*}
If $C=\widetilde{aC_0+bF}$ for non-negative $a$ and $b$  and some fibre $F$, then
\begin{align*}
K_{\Yc}.(\widetilde{aC_0+bF})&=\left(f^*\left(-C_0-2F\right)+E_1+\dfrac{1}{2}\sum_{i=1}^3a_iE_i\right)(\widetilde{aC_0+bF})\\
&=\underbrace{\left(f^*\left(-C_0-2F\right)+E_1+\dfrac{1}{2}\sum_{i=1}^3a_iE_i\right)(\widetilde{aC_0})}_{=0}\\
&+\underbrace{\left(f^*\left(-C_0-2F\right)+E_1+\dfrac{1}{2}\sum_{i=1}^3a_iE_i\right)(\widetilde{bF})}_{<0}<0.
\end{align*}
Finally case 2, where $\Sigma\subset D$, follows from Proposition \ref{Y adP M} with $M=0$.
\end{proof}

 The classification for $K_{\Yc}$-zero curves is as the following.
\begin{theorem}\label{canonical del Pezzo orders by F2}
 Let $\Wc$ be a minimal TAdPO over $\Fb_2$ with ramification divisor $D\equiv 2C_0+4F$ and ramification degree $e$. Also let $g:\Yc\rightarrow \Wc$ represent the blowups of $\Wc$. Then each of the followings gives  a $K_{\Yc}$-zero curve $E$ such that if $f:\Yc\rightarrow\Xc$ contracts $E$, then the order $\Xc$ is a CdPO.
\begin{enumerate}
 \item 
  The section $C_0$
 \item
  Blowing up  points $p\notin D$, $E:=F$ where $F$ is the fibre passing $p$.
  \item
  Blowing up  a set of points $\Sigma=\{p_1,\cdots, p_n\}\subset D$ in almost general position where $n\leq7$,  $E:=F$ is any fibre with multiplicity $2$ at $\Sigma$.
 \item
 $E:=E_j$ where $E_j$ is $(-2)$-exceptional curve.
 \end{enumerate}
 Note that all the blowups are assumed to satisfy conditions of Proposition \ref{almost general theorem F2}.
\end{theorem}
\begin{proof}
Cases 1 and 3 follow from Proposition \ref{Z iff X a}. Case 2 is as follows.
\begin{align*}
K_{\Yc}.\tilde{F}&=\left(f^*\left(-C_0-2F\right)+E_1\right)\tilde{F}\\
&=-1+E_1\tilde{F}=0.
\end{align*}
For Case 4 let $f:\Yc\rightarrow\Wc$ represent two blowups at a point $p$. The $p$ has to be in $D$. Let $E_1$ and $E_2$ be the corresponding blowups where $E_1^2=-1$ and $E_2^2=-2$. Then
\begin{align*}
K_{\Yc}.E_2&=\left(f^*\left(-C_0-2F\right)+\frac{2}{e}E_1+\dfrac{1}{e}E_2\right)E_2=0
\end{align*}
\end{proof}

\bibliographystyle{amsplain}

\end{document}